%% file: main.tex
\documentclass[11pt]{article}
\usepackage{epsfig,epsf,fancybox}
\usepackage{amsmath}
\usepackage{mathrsfs}
\usepackage{amssymb}
\usepackage{graphicx}
\usepackage{color}
\usepackage{multirow}
\usepackage{paralist}
\usepackage{verbatim}
\usepackage{galois}
\usepackage{algorithm}
\usepackage[toc,page,header]{appendix}
\usepackage{titletoc}
\usepackage{times}  
\usepackage{helvet}  
\usepackage{courier}  
\usepackage[hyphens]{url}  
\usepackage{natbib}  
\usepackage{caption} 
\usepackage[breaklinks=true]{hyperref}

\usepackage{algorithm}
\usepackage{algorithmic}
\usepackage{amsthm}
\usepackage{amsfonts}
\usepackage{newfloat}
\usepackage{listings}
\DeclareCaptionStyle{ruled}{labelfont=normalfont,labelsep=colon,strut=off} 
\floatstyle{ruled}
\newfloat{listing}{tb}{lst}{}
\floatname{listing}{Listing}

\usepackage{mathtools}
\usepackage{amsmath}
\usepackage{booktabs}
\usepackage{cleveref}
\usepackage{boxedminipage}
\usepackage{accents}
\usepackage{stmaryrd}
\usepackage[table]{xcolor}
\usepackage{hhline}
\usepackage{subfig}
\usepackage{bm}

\newif\ifcomments
\commentstrue 
\ifcomments
  \newcommand{\comm}[2][]{\textcolor{red}{\textbf{[#1:} #2\textbf{]}}}
  \newcommand{\qd}[1]{\comm[QD]{#1}}
  \newcommand{\zl}[1]{\comm[ZL]{#1}}
\else
  \newcommand{\comm}[2][]{}
  \newcommand{\qd}[1]{}
  \newcommand{\zl}[1]{}
  \newcommand{\xx}[1]{}
\fi

\providecommand{\assumptionname}{Assumption}
\providecommand{\corollaryname}{Corollary}
\providecommand{\definitionname}{Definition}
\providecommand{\lemmaname}{Lemma}
\providecommand{\propositionname}{Proposition}
\providecommand{\remarkname}{Remark}
\providecommand{\theoremname}{Theorem}
\providecommand{\examplename}{Example}

\newtheorem{defn}{\protect\definitionname}
\newtheorem{prop}{\protect\propositionname}
\newtheorem{assumption}{\protect\assumptionname}
\newtheorem{thm}{\protect\theoremname}

\newtheorem{lem}{\protect\lemmaname}
\newtheorem{rem}{\protect\remarkname}
\newtheorem{exam}{\protect\examplename}

\renewcommand\algorithmiccomment[1]{%
  \hfill\ $\vartriangleright$\ \eqparbox{COMMENT}{#1}%
}
\floatstyle{ruled}
\newfloat{algorithm}{tbp}{loa}
\providecommand{\algorithmname}{Algorithm}
\floatname{algorithm}{\protect\algorithmname}

\numberwithin{equation}{section}

\title{Revisiting Randomized Smoothing: Nonsmooth Nonconvex Optimization Beyond Global Lipschitz Continuity}
\author{Jingfan Xia\thanks{jf.xia@163.sufe.edu.cn, Shanghai University of Finance and Economics} \quad\quad\quad Zhenwei Lin\thanks{zhenweilin@163.sufe.edu.cn, Shanghai University of Finance and Economics} \quad\quad\quad Qi Deng\thanks{ qdeng24@sjtu.edu.cn, Shanghai Jiao Tong University} }
\date{}

\begin{document}
\include{macros}

\maketitle

\begin{abstract}

Randomized smoothing is a widely adopted technique for optimizing nonsmooth objective functions. However, its efficiency analysis typically relies on global Lipschitz continuity, a condition rarely met in practical applications. 
To address this limitation,  we introduce a new subgradient growth condition that naturally encompasses a wide range of locally Lipschitz functions, with the classical global Lipschitz function as a special case.
Under this milder condition, we prove that randomized smoothing yields a differentiable function that satisfies certain generalized smoothness properties. To optimize such functions, we propose novel randomized smoothing gradient algorithms that, with high probability,  converge to $(\delta, \epsilon)$-Goldstein stationary points and achieve a sample complexity of $\tilde{\mathcal{O}}(d^{5/2}\delta^{-1}\epsilon^{-4})$. By incorporating variance reduction techniques, we further improve the sample complexity to $\tilde{\mathcal{O}}(d^{3/2}\delta^{-1}\epsilon^{-3})$, matching the optimal $\epsilon$-bound under the global Lipschitz assumption, up to a logarithmic factor. Experimental results validate the effectiveness of our proposed algorithms.
\end{abstract}

\section{Introduction}

In this paper, we consider the following optimization problem 
\begin{equation}\label{pb:main}
\min_{\xbf\in\Rbb^{d}}\quad f(\xbf),
\end{equation}
 where $f:\Rbb^{d}\to\Rbb$ is a locally Lipschitz continuous function. Nonsmooth nonconvex problems formulated as \eqref{pb:main} have been prevalent in the machine learning and engineering areas~\citep{jain2017non,cui2021modern}. 
We assume that first-order information is unavailable. Instead, we can only query a zeroth-order oracle that returns function values.
This is motivated by numerous real-world applications where computing the exact (sub)gradient is difficult or even intractable, such as adversarial attacks on learning models~\citep{chen2017zoo}, simulation modeling~\citep{rubinstein2016simulation}, and fine-tuning large language models~\citep{malladi2023fine}.

Randomized smoothing~\citep{nesterov2017random} has emerged as a promising approach for solving our target problem. The key idea is to replace the standard (sub)gradient of $f(\xbf)$ with an unbiased estimate derived from function values at randomly perturbed points. This estimate corresponds to the gradient of a smooth approximation of $f(\xbf)$ via convolution operators.
Compared to classical derivative-free and black-box optimization methods~\citep{larson2019derivative, nelder1965simplex}, randomized smoothing is more scalable to high-dimensional data and offers stronger theoretical convergence guarantees. A substantial body of work has studied randomized smoothing methods for convex, nonconvex, and stochastic optimization~\citep{ghadimi2013stochastic, duchi2012randomized, balasubramanian2022zeroth, lin2022gradient}.
Despite this progress, most existing analyses rely heavily on the global Lipschitz property, which assumes the existence of a constant $L_f>0$ such that 
\begin{equation}\label{eq:global-lip}
\abs{f(\xbf)-f(\ybf)}\le L_f\norm{\xbf-\ybf}
\end{equation}
holds for any $\xbf,\ybf\in\Rbb^d$.
While this assumption simplifies convergence analysis, it is often unrealistic in practice. For instance, piecewise nonlinear functions (e.g., quadratic functions) may have unbounded subgradient norms, thereby violating~\eqref{eq:global-lip}.

This raises a natural question: \emph{Can we develop randomized smoothing algorithms for general nonsmooth, nonconvex problems when global Lipschitz continuity does not hold?}
Addressing this question introduces two primary challenges: First, we must appropriately model the local continuity properties of $f(\xbf)$ and understand how they affect the randomized smooth approximation functions. Second, we need to design efficient algorithms capable of solving the smooth approximation problem under these more relaxed assumptions.

In this paper, we aim to address these two difficulties and thereafter summarize our main contributions as follows. 

First, we introduce a subgradient growth condition that implies a generalized form of Lipschitz continuity, where the parameter can vary as a general function of the variable. Our condition is strictly weaker than global Lipschitz continuity, a widely used setting in nonsmooth nonconvex optimization. Furthermore, we demonstrate that applying randomized smoothing to a generalized Lipschitz continuous function yields a differentiable function that exhibits certain generalized smoothness properties. This provides the foundation for developing gradient-free algorithms and conducting complexity analysis.

Second, we propose and analyze several gradient-free algorithms based on randomized smoothing, specifically designed for optimizing the proposed generalized Lipschitz continuous functions.
In the Randomized Smoothing-based Gradient-Free (RS-GF) method, we introduce a step size that adapts to local smoothness and noise levels. We show that this algorithm achieves a convergence rate of \(\til{\Ocal}(d^{5/2}\delta^{-1}\epsilon^{-4})\) for reaching $(\delta,\epsilon)$-Goldstein stationary point.
We further propose a Randomized Smoothing-based batched Normalized Gradient-Free (RS-NGF) method, which employs a normalized stepsize to improve the complexity dependence on $d$.
Finally, using variance reduction techniques, we develop the Randomized Smoothing-based Normalized Variance Reduction Gradient-Free (RS-NVRGF) method, which,  with high probability, achieves a $(\delta,\epsilon)$-Goldstein stationary point with a complexity of \(\til{\Ocal}(d^{3/2}\delta^{-1}\epsilon^{-3})\). This result on the $\delta,\epsilon$ term, is nearly optimal for nonsmooth, nonconvex problems with global Lipschitz continuity, differing only by logarithmic factors from the best-known complexity bounds.

\subsection{Related works}

\textbf{Randomized smoothing}
The efficiency of randomized smoothing-based gradient methods have been extensively studied over the past decade~\citep{nesterov2017random, duchi2012randomized, flaxman2005online, ghadimi2013stochastic}, primarily focusing on two cases: (1) convex optimization and (2) nonconvex but smooth optimization problems. For a comprehensive review, we refer the readers to~\citet{liu2020primer}.
The extension of randomized smoothing to more general Lipschitz continuous functions has been explored in recent work~\citep{lin2022gradient, cui2023complexity}.
Specifically, \citet{lin2022gradient} demonstrated that, with spherical smoothing, a $(\delta, \vep)$-Goldstein stationary point can be obtained, using at most $\Ocal(d^{3/2}L_f^4\vep^{-4}+ d^{3/2}\Delta L_f^3 \delta^{-1}\vep^{-4})$ calls of the zeroth-order oracle. Subsequently, \citet{chen2023faster} employed variance reduction techniques to further improve the complexity to $\Ocal(d^{3/2}L_f^3 \vep^{-3}+ d^{3/2}\Delta L_f^2 \delta^{-1}\vep^{-3})$. In the later work, \citet{kornowski2023algorithm} further reduced dimensionality dependence from $\Ocal(d^{3/2})$ to $\Ocal(d)$, bringing the complexity bound in line with results for smooth or convex optimization. 

\textbf{Local Lipschitz continuity}:  The challenge of handling functions that lack global Lipschitz continuity has been primarily addressed in the context of convex optimization~\citep{grimmer2019convergence, lu2019relative, mai2021stability} or nonconvex optimization under certain regularity conditions, such as weak convexity~\citep{zhu2023unified, gao2024stochastic}. Typically, these works assume that the function or the norm of its subgradient satisfies specific growth conditions.
Closely related to our work,  \citet{lei2024subdifferentially} focus on the case where the growth condition of the subgradient follows a polynomial function under the quadratic growth condition. With Gaussian smoothing, they demonstrate the convergence rate to achieve the Goldstein stationary point in expectation, whereas we focus on the high-probability bound.

\textbf{Other related works}
The technique of gradient normalization and variance reduction has been playing an important role in our algorithms.  Most work on gradient normalization
has been focusing on the smooth problem or generalized smooth problem~\citep{cutkosky2020momentum,chen2023generalized}.
Very recently, \citet{hubler2024gradient} extended normalized stochastic gradient descent (NSGD) to problems with heavy-tailed noise, eliminating the need for clipping, and provided a convergence analysis in both expectation and high probability. Additionally, \citet{liu2024nonconvex, sun2024gradient} also demonstrated that NSGD can effectively handle heavy-tailed noise and the generalized smooth setting without requiring clipping.

Variance reduction techniques are widely used to improve the convergence rate by reducing the variance in gradient estimates. While ~\citet{allen2016variance} first introduced the variance reduction technique to nonconvex optimization, \citet{fang2018spider,zhou2020stochastic} improve the convergence rate by sampling a large batch periodically. \citet{cutkosky2019momentum} proposed a momentum-based variance reduction technique, which avoids generating a very large batch of samples. While most of the work on variance reduction has considered the case where the objective function is smooth, \cite{reisizadeh2023variance} tailored clipping and variance reduction techniques under the generalized smooth condition, achieving order-optimal complexity.

\subsection{Organization}

The paper is organized as follows. 
In Section 2, we present our new concept of generalized Lipschitz continuity and analyze the properties of smooth approximation. Section 3 describes three proposed algorithms along with their convergence properties. Section 4 provides numerical results to validate our approach. Finally, Section 5 concludes the paper and discusses potential directions for future work.

\subsection{Notation and preliminaries\label{sec:Notation-and-preliminaries}}

Throughout the paper, we use bold letters such as $\xbf,\ybf$ to present vectors. We use \(\norm{\cdot}\) to express the $\ell_2$ norm,
and use $\Bbb_{\delta}(\xbf)=\{\ybf: \norm{\ybf-\xbf}\le \delta\}$
to denote a ball around $\xbf$ with radius $\delta$, and use $\mbb S^{d-1}$ to denote the sphere of $\Bbb_1(0)$. We use $\Ubf(\Xbb)$ to represent the uniform distribution on the set $\Xbb$.
$\Vol$ are
used to show the volume. $\conv$ are used to present the convex hull. The $\tilde{\Ocal}$ notation
suppresses logarithmic factors for simplicity.
For brevity, we use the notation $[T]=\{1,2,\ldots, T\}$.

The Clarke subdifferential~\citep{clarke1990optimization} is defined as
$\partial f(\xbf)\coloneqq\conv\{\lim \nabla f(\xbf_k): \xbf_k\rightarrow \xbf, \xbf_k\notin S,\xbf_k\in D\}$ where $S$ is any set of measure zero, and $D$ is the set of differentiable points (of measure one).
Computing approximate stationary points regarding Clarke subdifferential has proven to be intractable \citep{zhang2020complexity}. Alternatively, we resort to 
the Goldstein $\delta$-subdifferential ($\delta>0$), which is defined as $\partial_{\delta}f(\xbf)\coloneqq\conv(\cup_{\ybf\in\Bbb_{\delta}(\xbf)}\partial f(\ybf))$.
We say that  $\xbf$  is an $\epsilon$-stationary point if
$\min\{\norm{\gbf}:\gbf\in\partial f(\xbf)\}\leq\epsilon$.
We say that $\xbf$ is a $(\delta,\epsilon)$-Goldstein-stationary point if
$\min\{\norm{\gbf}:\gbf\in\partial_{\delta}f(\xbf)\}\leq\epsilon$. Throughout the rest of the paper, we assume $f(\xbf)$ is lower bounded by $f^{\star}$.

\section{Generalized Lipschitz continuity \label{sec:generalized-Lipschitz-continuit}}


A globally Lipschitz‑continuous function enjoys a uniform bound on the norm of every subgradient. For non-Lipschitz functions, it is therefore natural to describe their local Lipschitz continuity through a local subgradient growth condition.  
Thus, we begin with a class of functions for which the upper bound and variation of the subgradient with respect to can be characterized as follows.

Let $\alpha:\Rbb^{d}\to\Rbb_{+}$ be continuous, and let  $\beta:\Rbb^{d}\times\Rbb_{+}\to\Rbb_{+}$ be continuous and assume $\beta(\xbf,r)$ to be non‑decreasing in $r$. 
The continuous function $f(\cdot):\Rbb^{d}\to\Rbb$  satisfies an $(\alpha,\beta)$  subgradient growth condition if  
\[
\sup_{\zeta\in\partial f(\xbf)}\|\zeta\|\leq\alpha(\xbf),\quad \forall \xbf\in\Rbb^d,
\]
 and the variation of $\alpha(\cdot)$ can be bounded by
\begin{equation}
\abs{\alpha(\xbf)-\alpha(\ybf)}\leq\beta(\xbf,\norm{\ybf-\xbf}),\quad \forall \xbf,\ybf\in\Rbb^{d}.\label{eq:lipschitz_alpha}
\end{equation}
The first condition provides a general local subgradient upper bound controlled by \(\alpha(\xbf)\), while the second condition controls the variation between the subgradients using the distance between points. Putting them together, we give the local Lipschitz continuity of $f(\xbf)$ in the following lemma. The proof can be found in Appendix~\ref{sec:proof_for_def}.
\begin{lem}
\label{lem:lipschitz-1}Suppose that $f:\Rbb^{d}\to\Rbb^{m}$
has an $(\alpha$,$\beta$) subgradient growth condition. Then, we have
\[
\norm{f(\xbf)-f(\ybf)}\leq(\alpha(\xbf)+\beta(\xbf,\norm{\ybf-\xbf}))\norm{\xbf-\ybf}.
\]
\end{lem}

\begin{rem}
Although we show that the subgradient growth condition implies generalized Lipschitz continuity,  one may first start from a generalized Lipschitz assumption and derive the corresponding subgradient bounds. See Appendix~\ref{sec:discussion_def} for more discussion. 
We emphasize the subgradient upper‑bound formulation here because it is easy to estimate from first‑order information, which will facilitate a cleaner characterization of the gradient‑Lipschitz constants of the smoothed objective.
\end{rem}

We now present several examples to illustrate that the proposed subgradient growth condition can accommodate a wide range of important functions.
\begin{exam}
If $f(\xbf)$ is $L$-Lipschitz continuous, then it has $(\alpha,\beta)$ subgradient growth condition with $\alpha(\cdot)=L$ and
$\beta(\cdot,\cdot)=0$.
\end{exam}
\begin{exam}
If $f(\xbf)$ is differentiable  and its gradient $\nabla f(\xbf)$ is  $L$-Lipschitz continuous, then $f(\xbf)$ has $(\alpha,\beta)$ subgradient growth condition
with $\alpha(\xbf)=\norm{\nabla f(\xbf)}$ and
$\beta(\xbf,r)=Lr$.

\end{exam}
\begin{exam}
If \(f(\xbf)\) satisfies the generalized smoothness condition $\norm{\nabla f(\xbf)-\nabla f(\ybf)}\leq \phi(\norm{\nabla f(\xbf)},\norm{\xbf-\ybf})\norm{\xbf-\ybf}$ where $\phi(x,y):\Rbb_+\times\Rbb_+\to\Rbb_+$ is non-decreasing on both $x$ and $y$ as defined in~\citep{chen2023generalized,li2024convex,tyurin2024toward}, then $f$ has $(\alpha,\beta)$ subgradient growth condition with $\text{\ensuremath{\alpha(\cdot)}=}\norm{\nabla f(\xbf)}$
and $\text{\ensuremath{\beta(\xbf,r)}=}\phi(\norm{\nabla f(\xbf)},r)r.$

\end{exam}
\begin{exam}\label{ex:non_decreasing}
If $\sup_{\zeta\in\partial f(\xbf)}\|\zeta\|\leq\gamma(\norm{\xbf})$
where $\gamma$ is non-decreasing on $\norm{\xbf}$ and is finite
on any bounded set of $\norm{\xbf}$, then $f$ has $(\alpha,\beta)$ subgradient growth condition with $\beta(\xbf,r)=$ $\abs{\alpha(\norm{\xbf})+\alpha(\norm{\xbf}+r)}$. 

\end{exam}
\begin{exam}\label{ex:add}
Consider the composite problem $f(x)=g(\xbf)+h(\xbf)$ where $g:\Rbb^{d}\to\Rbb$
has $(\alpha_{g},\beta_{g})$ subgradient growth condition and $h:\Rbb^{d}\to\Rbb$
has $(\alpha_{h},\beta_{h})$ -subgradient growth condition, then
$f$ has $(\alpha_{f},\beta_{f})$ -subgradient growth condition with $\text{\ensuremath{\alpha_{f}(\xbf)=\alpha_{h}(\xbf)+\alpha_{g}(\xbf)} }$
and $\beta_{f}(\xbf,r)=\beta_{g}(\xbf,r)+\beta_{h}(\xbf,r)$.
\end{exam}

\paragraph{Non-polynomial growth} Compared to the definitions presented in existing works, our definition is more general and flexible. For instance, it encompasses all polynomial subgradient bounds (see Example~\ref{ex:non_decreasing}) that have been proposed in the literature \cite{zhu2023unified, lei2024subdifferentially}. Furthermore, our framework extends beyond these cases to include more general, non-polynomial growth functions.
To illustrate this, consider the function \( f(x) = \exp(\abs{x}) - x^3 \), whose subdifferential is given by:
\[\partial f(x)=\begin{cases}
\sign(x)\exp(\abs x)-3x^2 & x\neq0,\\{}
[-1,1] & x=0.
\end{cases}\]
In this case, let  $\alpha(x)=\exp(\abs x)+3x^2$. Clearly, the subgradient can not be bounded by a polynomial of $\norm{\xbf}$. This choice satisfies the condition in Example~\ref{ex:non_decreasing}.

\paragraph{The composite function}
Consider the composite problem \[f(x)=g\circ\hbf(\xbf)\] where $g:\Rbb^{m}\to\Rbb$
and $\hbf:\Rbb^{d}\to\Rbb^{m}$
are Lipschitz continuous.
Since $f(\xbf)$ is not Clarke regular, the standard chain rule in subdifferential computation does not hold~\citep{bagirov2014introduction}. 
Fortunately, by exploiting the local Lipschitzness of $g$ and $\hbf$, we can still apply the chain rule to estimate the local Lipschitz continuity of $f(\xbf)$. 
\begin{prop}\label{ex:composite}
Suppose that  $g:\Rbb^{m}\to\Rbb$
has $(\alpha_{g},\beta_{g})$ subgradient growth condition and $\hbf:\Rbb^{d}\to\Rbb^{m}$
has $(\alpha_{\hbf},\beta_{\hbf})$ subgradient growth condition,
then $f$ has $(\alpha_{f},\beta_{f})$ subgradient growth condition with $\text{\ensuremath{\alpha_{f}(\xbf)=\alpha_{\hbf}(\xbf)\alpha_{g}(\hbf(\xbf))} }$
and $\beta_{f}(\xbf,r)=(\alpha_{\hbf}(\xbf)+\beta_{\hbf}(\xbf,r))\beta_{g}(h(\xbf),(\alpha_{\hbf}(\xbf)+\beta_{\hbf}(\xbf,r))r)+\beta_{\hbf}(\xbf,r)\alpha_{g}(h(\xbf))$.
\end{prop}

\section{Randomized smoothing}

Randomized smoothing uses a convolution kernel to construct a smooth approximation of the original nonsmooth function. We consider spherical smoothing: 
\begin{equation}\label{def:randomized-smoothing-1}
f_{\delta}(\xbf)=\Expe_{\wbf\sim \Ubf(\Bbb_1(0))}f(\xbf+\delta\wbf).
\end{equation}
When $f(\xbf)$ is globally Lipschitz continuous as in \eqref{eq:global-lip}, \citet{lin2022gradient} have shown that the smoothing function $f_\delta(\xbf)$ is globally  Lipschitz smooth.  
However, this result is not applicable when \eqref{eq:global-lip} fails to hold. 
In this subsection, we will establish some useful properties, including some new generalized smoothness conditions of $f_\delta(\xbf)$.  The omitted proof can be found in the Appendix \ref{sec:proof_for_def}. 

First, we will show that the smoothed function, as well as its gradient, is well-defined. Furthermore, we will show the relationship between optimizing the surrogate function and optimizing the original function.
\begin{lem}\label{lem:final}
Suppose that $f$ has $(\alpha, \beta)$ subgradient growth condition. Then $f_{\delta}$, defined by \eqref{def:randomized-smoothing-1},  is continuously differentiable,  
and $\nabla f_{\delta}(\xbf)=\Expe_{\wbf}[\nabla f(\xbf+\delta\wbf)]$.
Furthermore, we have $\nabla f_{\delta}(\xbf)\in\partial_{\delta}f(\xbf)$.
\end{lem}

The following proposition establishes the properties of $f_{\delta}$, especially the smoothness of $f_{\delta}$.
\begin{prop}\label{prop:smoothness}
Let $f: \mathbb{R}^d \to \mathbb{R}$ have $(\alpha, \beta)$ subgradient growth condition, and $f_{\delta}$ be given by \eqref{def:randomized-smoothing-1}. Then, we have the following results:

1. $\left| f_{\delta}(\xbf) - f(\xbf) \right| \leq \delta \left( \alpha(\xbf) + \beta(\xbf, \delta) \right)$.

2.  The function $f_\delta$ has $(\tilde{\alpha},\tilde{\beta})$ subgradient growth condition with $\tilde{\alpha}(\xbf)=\alpha(\xbf)+\beta(\xbf,\delta)$ and $\tilde{\beta}(\xbf,r)=\beta(\xbf,r+\delta)$. Furthermore, \[\left| f_{\delta}(\xbf) - f_{\delta}(\ybf) \right| \leq \left( \alpha(\xbf) + \beta(\xbf, \|\xbf - \ybf\| + \delta) \right) \|\xbf - \ybf\|.\]

3.  $f_\delta$ is locally Lipschitz smooth, i.e. \[
\begin{split}
    &\left\| \nabla f_{\delta}(\xbf) - \nabla f_{\delta}(\ybf) \right\| \\
    &\leq \frac{c\sqrt{d}}{2\delta} \Big( 2\alpha(\xbf) + \beta(\xbf, \delta)  + \beta(\xbf, \|\xbf - \ybf\| + \delta) \Big) \|\xbf - \ybf\|.
\end{split}\]
\end{prop}
Consequently, we have the following useful descent property.
\begin{lem}
\label{lem:smoothness_descent}Under the assumptions of Proposition~\ref{prop:smoothness},  for any $\xbf,\ybf \in \Rbb^{d}$, the following inequality holds:
\begin{equation}
f_{\delta}(\ybf)\leq f_{\delta}(\xbf)+\inner{\nabla f_{\delta}(\xbf)}{\ybf-\xbf}+\frac{\ell(\xbf,\norm{\ybf-\xbf})}{2}\norm{\ybf-\xbf}^{2},\label{eq:smooth_descent_lemma}
\end{equation}
where $\ell(\xbf,r)=\frac{c\sqrt{d}(2\alpha(\xbf)+\beta(\xbf,\delta)+\beta(\xbf,r+\delta))}{2\delta}.$
\end{lem}
\begin{rem}
Proposition~\ref{prop:smoothness} implies that \( f_\delta(\xbf) \) satisfies some non-standard local Lipschitz smoothness. Notably, this condition also differs from the generalized smoothness ($(L_0, L_1)$-smoothness) framework~\citep{zhang2019gradient, chen2023generalized, li2024convex, tyurin2024toward}, which typically bounds the gradient Lipschitz continuity as a function of the gradient norm. In contrast, our bound depends more directly on \(\xbf\), providing a different source of non-Lipschitz smoothness.
 Moreover, our local smoothness condition can be seen as a computable and concrete instance of directional smoothness as studied in \citep{mishkin2024directional}, where $\ell(\xbf,\norm{\ybf-\xbf})$ serves as the directional smoothness function.
\end{rem}
 In the following part, we use the notation $\ell(\xbf,r)$ to show the local smoothness of $f_\delta$.
We always estimate the gradient of $f_{\delta}$ at the point $\xbf$ with  a batch of i.i.d. samples $S_B=\{\wbf_{i}\}_{i=1}^{B}$, where $\wbf_{i}$ is drawn from $\Ubf(\mathbb{S}^{d-1})$, we estimate the gradient as
\begin{equation}\label{eq:grad_esti_batch}
    \gbf(\xbf,S_B) =\frac{1}{B}\sum_{i=1}^{B}\frac{(F(\xbf+\delta\wbf_{i})-F(\xbf-\delta\wbf_{i}))d}{2\delta}\wbf_{i}.
\end{equation}
When $B=1$, we denote it as $g(\xbf,\wbf)$. In the Appendix \ref{subsec:gradient_estimator_property}, we show the detailed properties of this gradient estimator.

While we primarily focus on central difference gradient estimators, it is also possible to use forward or backward difference estimators, which can reduce the number of function evaluations by nearly half. The subsequent analysis remains the same in either case. However, a key advantage of using central difference estimators is their ability to significantly mitigate the dependence of estimation variance on the dimensionality $d$~(see \citet{shamir2017optimal}).


\section{Algorithms and convergence analysis\label{sec:Algorithms-and-Convergence}}
In the standard complexity analysis for nonsmooth optimization, the bound on the convergence rate typically involves some Lipschitz parameters over the solution sequence.
When the Lipschitz continuity varies with $\xbf$ in an unbounded domain, 
a major challenge is to ensure the boundedness of  $\{\xbf_t\}$ in the trajectory. 
 To address this issue, it is natural to impose some relation between $f(\xbf)$ and the growth of $\norm{\xbf}$.
 Specifically, we require the following level-boundedness assumption.
 \begin{assumption}\label{assu:level_set}
$f_\delta(\xbf)$ is level-bounded, namely, for any $\nu\in\Rbb$, the level set 
$\Lcal_{\nu}=\{\xbf:f_\delta(\xbf)\leq\nu\}$ is  bounded. 
\end{assumption}
The level-boundedness is essential in developing iteration complexity for nonsmooth nonconvex optimization. It implies that by bounding $f_\delta(\cdot)$, we can control the Lipschitz terms involving $\alpha(\xbf)$ and $\beta(\xbf, \delta)$.
Level-boundedness is satisfied in the coercive function, which exhibits high nonlinearity and hence is natural in our setting. 
For example, \citet{lei2024subdifferentially} assumes a quadratic growth condition, which can be seen as a special case of Assumption~\ref{assu:level_set}. Some sufficient conditions about the assumption can be found in Appendix \ref{subsec:discuss_assum}. 



\subsection{The RS-GF method}
First, we present the randomized-smoothing-based gradient-free method (RS-GF method) in Algorithm~\ref{alg:sgd}. Similar to~\citet{lei2024subdifferentially}, we employ a stepsize adapted to the local Lipschitz continuity. 
We state an informal convergence result for brevity and leave the formal theorem and its proof in Appendix \ref{sec:sgd}.
\begin{algorithm}
\caption{RS-GF\label{alg:sgd}}
\begin{algorithmic}[1]
\STATE{\textbf{Input:} initial state $\xbf_1$, parameters $T \in \mathbb{N}$, and stepsize $\{\eta_{t}\}_{t=1}^{T}$.}
\FOR{$t = 1, \ldots, T$}
    \STATE{\textbf{Draw} $\wbf_{t}\sim \Ubf(\mathbb{S}^{d-1})$} and compute $\gbf(\xbf_t,\wbf_t)$.
    \STATE{\textbf{Update:} $\xbf_{t+1} = \xbf_t - \eta_t \gbf(\xbf_t,\wbf_t)$.}
\ENDFOR
\end{algorithmic}
\end{algorithm}

\begin{thm}[Informal]
\label{thm:result_sgd}
Under Assumption~\ref{assu:level_set}, let $T=\tilde{\Ocal}(d^{5/2}\delta^{-1}\epsilon^{-4})$  and set the stepsize $\eta_{t} =\frac{1}{\log\left(2T/{p}\right)\sqrt{T}}C_{t},$ with $p\in(0,1)$ and 
\begin{equation}\label{eq:sgd_stepsize}
    \begin{aligned}
C_{t} = &\min\left\{ \frac{\Delta }{2(d+1)(c_{t,\delta}+1)^{2}},
\frac{\sqrt{6}}{12}\frac{\Delta }{\sigma(\xbf_{t})c_{t,\delta}}
\frac{1}{dc_{t,\delta}}\sqrt{\frac{\Delta }{\ell(\xbf_{t},\frac{\Delta }{d\sqrt{T}(c_{t,\delta}+1)})}}\right\},
    \end{aligned}
\end{equation}
where $\Delta >f_\delta(\xbf_1)-f^{\star}$ and $c_{t,\delta}=\alpha(\xbf_{t})+\beta(\xbf_{t},3\delta)$.
Then with probability at least $1-p$, the overall sample complexity to achieve $\min_{t\in[T]}\norm{\nabla f_{\delta}(\xbf_{t})}^{2}\leq\epsilon^{2}$
is $\tilde{\Ocal}(\epsilon^{-4}\delta^{-1}d^{\frac{5}{2}})$. Furthermore,  $\tilde{\xbf}\in\argmin_{\xbf_t,t\in[T]}\norm{\nabla f_{\delta}(\xbf_{t})}$ is a $(\delta,\epsilon)$-Goldstein stationary point of $f$.
\end{thm}
\begin{rem}
The stepsize \eqref{eq:sgd_stepsize} is employed to account for local Lipschitz continuity in algorithm design. 
The rationale behind this choice is to adapt dynamically to local gradient estimation errors (see Lemma~\ref{lem:variance}) and local smoothness variations (see Lemma~\ref{lem:smoothness_descent}), both of which evolve with the sequence \(\{\xbf_t\}\).
\end{rem}
\begin{rem}
Another challenge in the convergence analysis is ensuring the boundedness of the iterates. Without this condition, the sequence \(\{\xbf_t\}\) may become unbounded, leading to a growing Lipschitz parameter and, consequently, an arbitrarily small stepsize~\eqref{eq:sgd_stepsize}.
To address this issue, we employ a high-probability convergence analysis introduced by \citet{liu2023near, chezhegov2024gradient}. At a high level, we show that under the level-boundedness assumption, if \( f_\delta(\xbf_{t-1}) \) is bounded, then with high probability, \( f_\delta(\xbf_t) \) remains bounded as well. This guarantees that the entire sequence \(\{\xbf_t\}\) stays bounded with high probability.
For brevity, we defer the technical details to Appendix~\ref{sec:sgd}.
\end{rem}
\subsection{The RS-NGF method}
It is important to note that the complexity of Algorithm~\ref{alg:sgd} exhibits an undesirable dependence on the dimensionality, making it less favorable compared to the bounds established in \citet{lin2022gradient} and \citet{chen2023faster} for Lipschitz problems. 
To address this issue, we propose using a normalized step size, which scales the gradient estimator by its norm. We introduce the randomized smoothing-based normalized gradient-free (RS-NGF) method, detailed in Algorithm~\ref{alg:RS-NGF}. Below, we provide an informal statement of our main result, with the formal version and its corresponding proof available in Appendix~\ref{sec:nsgd}.

\begin{algorithm}
\caption{RS-NGF}
\label{alg:RS-NGF}
\begin{algorithmic}[1]
\STATE {\textbf{Input:} Initial state $\xbf_{1}$, parameter $T \in \mathbb{N}$, stepsize $\{\eta_{t}\}_{t=1}^{T}$, batch size $\{B_{t}\}_{t=1}^{T}$}.

\FOR{$t =  1, \ldots, T$}
    \STATE \textbf{Draw} a batch of samples $S_{B_{t}} = \{\wbf_{t,i} : i = 1, \dots, B_{t}\}$ and compute $\gbf(\xbf_t,S_{B_t})$ as in \eqref{eq:grad_esti_batch}.
    \STATE \textbf{Update:}
  $\xbf_{t+1} = \xbf_{t} - \eta_{t} \frac{\gbf(\xbf_t,S_{B_t})}{\|\gbf(\xbf_t,S_{B_t})\|}.$
\ENDFOR
\end{algorithmic}
\end{algorithm}

\begin{thm}[Informal]
\label{thm:result_nsgd}Under Assumption~\ref{assu:level_set}, let the iteration number $T$ be $\tilde{\Ocal}(d^{1/2}\delta^{-1}\epsilon^{-2})$, batch size be $\tilde{\Ocal}(Td^{\frac{1}{2}}\delta)$, the stepsize   
$\eta_{t}  =\frac{1}{\log(2T/p)\sqrt{T}}C_{t},$ with 
\begin{equation}\label{eq:stepsize_nsgd}
    \begin{aligned}
        C_{t} = &\min\left\{ \frac{\Delta }{12(c_{t,\delta} + 1)}, \sqrt{\frac{\Delta }{3\ell(\xbf_{t}, \frac{\Delta }{(c_{t,\delta} + 1)\sqrt{T}})}}, \frac{2\Delta }{\sqrt{\ell(\xbf_{t}, \frac{\Delta }{(c_{t,\delta} + 1)\sqrt{T}})}}\right\},
    \end{aligned}
\end{equation}
where $\Delta >f_\delta(\xbf_1)-f^{\star}$ and $c_{t,\delta}=\alpha(\xbf_{t})+\beta(\xbf_{t},\delta)$.
Then with a probability of at least $1-p$, 
the overall sample complexity to achieve $\min_{t\in[T]}\norm{\nabla f_{\delta}(\xbf_{t})}\leq\epsilon$
is $\tilde{\Ocal}(d^{3/2}\epsilon^{-4}\delta^{-1})$. Furthermore,  $\tilde{\xbf}\in\argmin_{\xbf_t,t\in[T]}\norm{\nabla f_{\delta}(\xbf_{t})}$ is a $(\delta,\epsilon)$-Goldstein stationary point of $f$.
\end{thm}
\begin{rem}
Similar to Algorithm~\ref{alg:sgd}, the stepsize choice~\eqref{eq:stepsize_nsgd} adapts to the locally varying gradient estimation errors and smoothness. However, by introducing the normalization term \(\|\gbf(\xbf_t, S_{B_t})\|\) in Step 4 of Algorithm~\ref{alg:RS-NGF}, we effectively reduce the dimensional dependence in setting the stepsize $\eta_t$ by a factor of $d$ . As a result, the complexity dependence on the dimension improves to \(\Ocal(d^{\frac{3}{2}})\), aligning with the complexity bounds established in \citet{lin2022gradient, chen2023faster}.
Similar to the Algorithm~\ref{alg:sgd}, the stepsize choice~\eqref{eq:stepsize_nsgd} adapts to the locally varying gradient estimation errors and smoothness. However, by introducing the normalization term \(\|\gbf(\xbf_t, S_{B_t})\|\) in Step 4 of Algorithm~\ref{alg:RS-NGF}, we effectively reduce the dimensional dependence in setting the stepsize $\eta_t$ by a factor of $d$ . As a result, the complexity dependence on the dimension improves to \(\Ocal(d^{\frac{3}{2}})\), aligning with the complexity bounds established in \citet{lin2022gradient, chen2023faster}.
\end{rem}

\subsection{The RS-NVRGF method}
It should be noted that the convergence rate with the $\vep$-term is suboptimal when $f(\xbf)$ is globally Lipschitz continuous~\citep{chen2023faster}. To further improve the rate, we apply the variance reduction technique (e.g. \citet{fang2018spider}) and use a normalized stepsize as in Algorithm~\ref{alg:RS-NGF} to adapt to local smoothness.
We describe the randomized smoothing-based gradient-free method with normalized stepsize and variance reduction (RS-NVRGF method) in Algorithm~\ref{alg:storm}. Here, we present an informal version of our result; the formal version and corresponding proof can be found in Appendix \ref{sec:nsvrg1}.


\begin{algorithm}
\caption{RS-NVRGF\label{alg:storm}}
\begin{algorithmic}[1]
\INPUT {Initial state $\xbf_{1}$, parameters $T \in \mathbb{N}$, batch size $b_t$ for $t\in[T]$, batch size $B_t$ for $t \bmod q = 1$, sample period $q$.}
\FOR{$t = 1, \ldots, T$}
    \IF{ $t \bmod q = 1$}
        \STATE{\textbf{Draw} $S_{B_t} = \{\wbf_{t,i}\}_{i=1}^{B_t}$ from a uniform distribution on the unit sphere $\mathbb{S}^{d-1}$ and compute $\gbf(\xbf_{t}, S_{B_t})$} as \eqref{eq:grad_esti_batch}.
        \STATE{\textbf{Compute} $\mbf_t = \gbf(\xbf_t, S_{B_t})$}.
    \ELSE
        \STATE{\textbf{Draw} $S_{b_t} = \{\wbf_{t,i}\}_{i=1}^{b_t}$ uniformly from the unit sphere $\mathbb{S}^{d-1}$ and calculate $\gbf(\xbf_{t}, S_{b_t})$ and $\gbf(\xbf_{t-1}, S_{b_t})$} as \eqref{eq:grad_esti_batch}.
        \STATE{Calculate $\mbf_t = \mbf_{t-1} + \gbf(\xbf_t, S_{b_t}) - \gbf(\xbf_{t-1}, S_{b_t})$}
        \STATE{Update $\xbf_{t+1} = \xbf_t - \eta_t \frac{\mbf_t}{\|\mbf_t\|}$}.
    \ENDIF
\ENDFOR
\end{algorithmic}
\end{algorithm}
\begin{thm}[Informal]
\label{thm:result_nsvrg}Under Assumption~\ref{assu:level_set}, 
 let $B_{s_{0}}=\left\lceil \frac{72\sigma^{2}(\xbf_{s_{0}})T\delta}{\sqrt{d}}\right\rceil $ in which  $t_{0}\bmod q=0$, $b_{s}=\left\lceil 72qd\right\rceil$,
$q=\left\lceil \epsilon^{-1}\right\rceil $ ,
the stepsize $\eta_{t}=\frac{1}{\log\left(2T/{p}\right)\sqrt{T}}C_{t},$
with $p\in(0,1)$ and
\begin{equation}
    \begin{aligned}
          C_{t}=&\min\left\{\frac{\delta^\frac{1}{2}}{d^\frac{1}{4}} ,\frac{\Delta }{24(c_{t,\delta}+1)},
          \frac{\Delta \delta^{\frac{1}{2}}}{d^\frac{1}{4}(c_{t,\delta} +\frac{\Delta }{\sqrt{T}(c_{t,\delta}+1)}))},\right .
  \left .\sqrt{\frac{2\Delta }{3\ell(\xbf_{t},\frac{\Delta }{(c_{t,\delta}+1)\sqrt{T}})}}\right\},
    \end{aligned}
\end{equation}
where $\Delta >f_\delta(\xbf_1)-f^{\star}$ and $c_{t,\delta}=\alpha(\xbf_{t})+\beta(\xbf_{t},\delta)$.
Then with probability of at least $1-p$, the overall sample complexity to achieve $\min_{t\in[T]}\norm{\nabla f_{\delta}(\xbf_{t})}\leq\epsilon$
is $\tilde{\Ocal}(d^{3/2}\epsilon^{-3}\delta^{-1})$. Then $\tilde{\xbf}\in\argmin_{\xbf_t,t\in[T]}\norm{\nabla f_{\delta}(\xbf_{t})}$ is a $(\delta,\epsilon)$-Goldstein stationary point of $f$.
\end{thm}

\section{Numerical experiment}\label{sec:Numerical-Experiments}
In this section, we present experiments that demonstrate the improvement in convergence performance achieved by our algorithms.

\textbf{Simulation Experiment.}  We consider a nonlinear optimization problem where, given a set of reference points \(\{\mathbf{a}_i\}_{i=1}^m \in \mathbb{R}^d\), a set of target measurements \(\{d_{ij}\}_{(i,j) \in N_x}\) and \(\{\tilde{d}_{ij}\}_{(i,j) \in N_a}\), and a transformation function \(r(\cdot)\), we aim to estimate the variables \(\mathbf{x}_1, \mathbf{x}_2, \ldots, \mathbf{x}_n \in \mathbb{R}^d\).  The optimization problem is formulated as:
\begin{align*}    \min_{\xbf_1,\xbf_2,\dots,\xbf_n}&
\frac{1}{\abs{N_x}+\abs{N_a}}\left(\sum_{(i,j)\in N_x} r(\left|\norm{\xbf_i-\xbf_j}-d_{ij}\right|)+\sum_{(i,j)\in N_a} r\left(\left|\norm{\abf_i-\xbf_j}-\tilde{d}_{ij}\right|\right)\right),
\end{align*}
where \(N_x\) and \(N_a\) denote sets of index pairs for variable-variable and reference-variable measurements, respectively, and \(r(\cdot)\) is a nonconvex non-Lipschitz continuous function. 
We vary 
$r(\cdot)$ to analyze how increasingly non-Lipschitz properties affect convergence and accuracy. 
Especially, we consider a high-order polynomial of function $(\cdot)^5$; while $\exp((\cdot)^3) $ exhibits exponential growth.  
\begin{figure*}[htb]
\begin{centering}
\begin{minipage}[t]{0.5\columnwidth}%
\includegraphics[width=8.3cm,height=5.8cm]{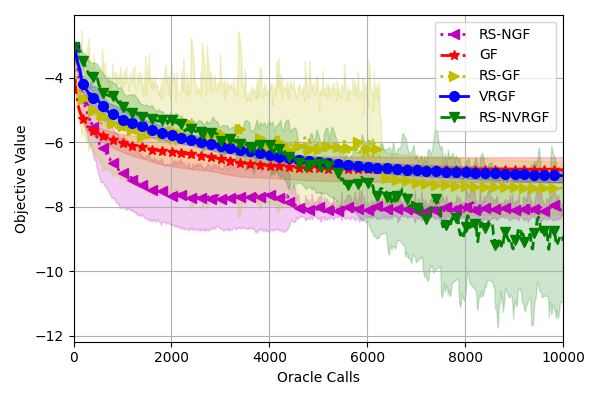}%
\end{minipage}\hfill
\begin{minipage}[t]{0.5\columnwidth}%
\includegraphics[width=8.3cm,height=5.8cm]{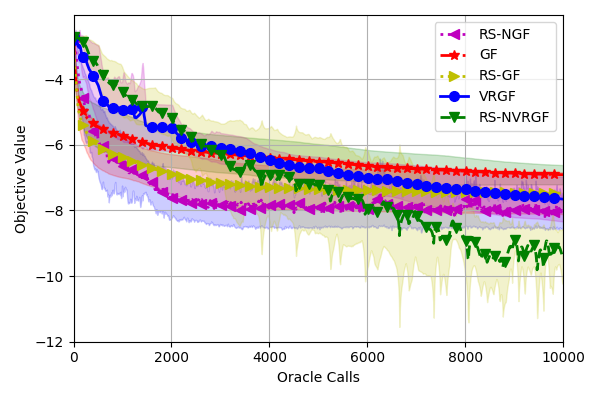}%
\end{minipage}
\end{centering}
\caption{\label{fig:SNL}The performance of the five algorithms is evaluated by plotting the logarithm of the objective value on the $y$-axis against the calls of zeroth-order oracles on the $x$-axis. The left-hand side is the result for $r(\cdot)=(\cdot)^5$, and the right-hand side is the result for $r(\cdot)=\exp((\cdot)^3)$. For the exponential loss, an additional logarithm is applied to enhance clarity in the visualization.}
\end{figure*}

We compare the performance of our proposed methods with the zeroth-order gradient descent method with constant stepsize (GF)~\citep{lin2022gradient} and zeroth-order variance
reduced gradient descent algorithms (VRGF) ~\citep{chen2023faster} during $10^{4}$ evaluations of zeroth-order oracles. 
We follow the setup in~\cite{nie2009sum}, with slight modifications. We generate $\{\xbf_i\}_{i=1}^{10}$ randomly from the unit square \( [-0.5, 0.5] \times [-0.5, 0.5] \). Additionally, there are $\{\abf_i\}_{i=1}^{4}$ located at \( (\pm 0.45, \pm 0.45) \). The initial solution is also generated randomly from the unit square \( [-0.5, 0.5] \times [-0.5, 0.5] \). For the case $r(\cdot)=\exp((\cdot)^3)$, we scale the points by dividing by 10.

We randomly select 100 variable-variable pairs and 50 reference-variable pairs. After removing any duplicates, 42 variable-variable pairs and 29 reference-variable pairs remain.
 $\alpha(\xbf)$ and $\beta(\xbf,r)$ are computed using the chain rule.
The step sizes are tuned from the set $\left\{2^{-2i+1}\right\}_{i=-6}^{6}.$
The batch sizes used in the Algorithm~\ref{alg:RS-NGF},  along with the small batch size $b$ used in the VRGF method and the Algorithm \ref{alg:storm},  are selected from:
$ \left[ 2, 4, 8, 16, 32 \right]$.
The parameter $q$ in Algorithm~\ref{alg:storm} is chosen from $\left[ 2, 4, 8, 16, 32 \right]$, and the large batch size $B$ used in the Algorithm~\ref{alg:storm} is set to $bq$.
The value of $\delta$ is set to $10^{-4}$. For each method, we performed 5 runs with different random seeds and plotted the average trajectory along with the standard deviation. 
The results are presented in Figure \ref{fig:SNL}. 

These results demonstrate that our proposed method achieves better convergence in both cases, as shown by the theoretical results. We observe that the methods using a constant step size require setting the step size to be very small to avoid divergence, which may lead to convergence to suboptimal solutions.

\textbf{Black-box Adversarial Attack.} 
\begin{table*}[h]
  \centering
    \begin{tabular}{l|llll}
    \toprule
    Method & Train Loss & Train Accuracy & Test Loss & Test Accuracy \\
    \midrule
    RS-GF & 4.416(0.311) & 0.463(0.020) & 5.593(0.329) & 0.433(0.017) \\
    VRGF  & 3.263(1.250) & 0.561(0.123) & 4.190(1.423) & 0.531(0.112) \\
    RS-NVRGF & 4.475(0.214) & 0.462(0.012) & 5.615(0.260) & 0.437(0.013) \\
    RS-NGF & \textbf{4.571(0.095)} & \textbf{0.452(0.007)} & \textbf{5.719(0.130)} & \textbf{0.427(0.008)} \\
    \bottomrule
    \end{tabular}%
      \caption{\label{tab:cifar10}Attacking results for CIFAR10 on ResNet.}
\end{table*}%
We investigate the black-box adversarial attack on image classification using ResNet. Our objective is to identify an adversarial perturbation that, when applied to an original image, induces misclassification in machine learning models while remaining imperceptible to human observers. This problem can be formulated as the following nonsmooth nonconvex optimization problem:
\begin{equation}
    \max_{\|\zeta\|_{\infty} \leq \kappa} \frac{1}{m} \sum_{i=1}^{m} \ell(a_i + \zeta, b_i),
\end{equation}
where the dataset \(\{(a_i, b_i)\}_{i=1}^m\) consists of image features \(a_i\) and corresponding labels \(b_i\), \(\kappa\) is the constraint level for the distortion, \(\zeta\) is the perturbation vector, and \(\ell(\cdot, \cdot)\) is the loss function. For the CIFAR10 dataset, we set \(\kappa = 0.1\). Following the experimental setup in \cite{chen2023faster} and \cite{lin2024decentralized}, we iteratively perform an additional projection step to ensure constraint satisfaction. 
We employ a pre-trained ResNet model with 99.6\% accuracy on the CIFAR10 dataset. 





To evaluate the effectiveness of the adversarial attack, we apply four optimization methods—VRGF, RS-NVRGF, RS-GF, and RS-NGF-to generate perturbations for 49,800 correctly classified images from the CIFAR10 training set. The adversarial perturbations are then evaluated on a separate test set of 10,000 CIFAR-10 images. Across all experiments, we set \(\delta = 0.001\).

For RS-NGF, the minibatch size \(B_t\) is tuned from \(\{16, 32, 64, 128, 256\}\). For RS-NVRGF and VRGF, the minibatch size \(b_t\) is tuned from \(\{16, 32, 64, 128, 256\}\), while the large batch size \(B_t\) is fixed at 49,800, corresponding to the total number of samples in the dataset, and the period parameter \(q\) is tuned from \(\{200, 400, 800, 1600, 3200\}\). For all algorithms, the step size \(\eta\) is tuned from \(\{0.0005, 0.001, 0.005, 0.01, 0.05, 0.1, 0.5\}\) and decayed by a factor of 0.8 every 100 epochs. The initial perturbation is set to zero for all experiments. 
 The results, averaged over five independent runs, are presented in Figure~\ref{fig:cifar10} and Table~\ref{tab:cifar10}.
\begin{figure*}[htb]
\begin{centering}
\begin{minipage}[t]{0.5\columnwidth}%
\includegraphics[width=8.3cm,height=5.8cm]{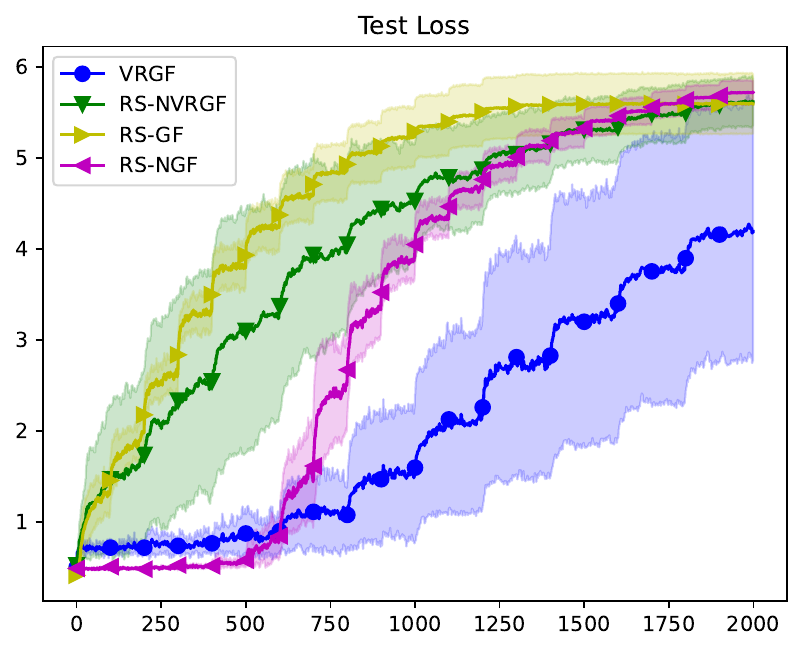}%
\end{minipage}\hfill
\begin{minipage}[t]{0.5\columnwidth}%
\includegraphics[width=8.3cm,height=5.8cm]{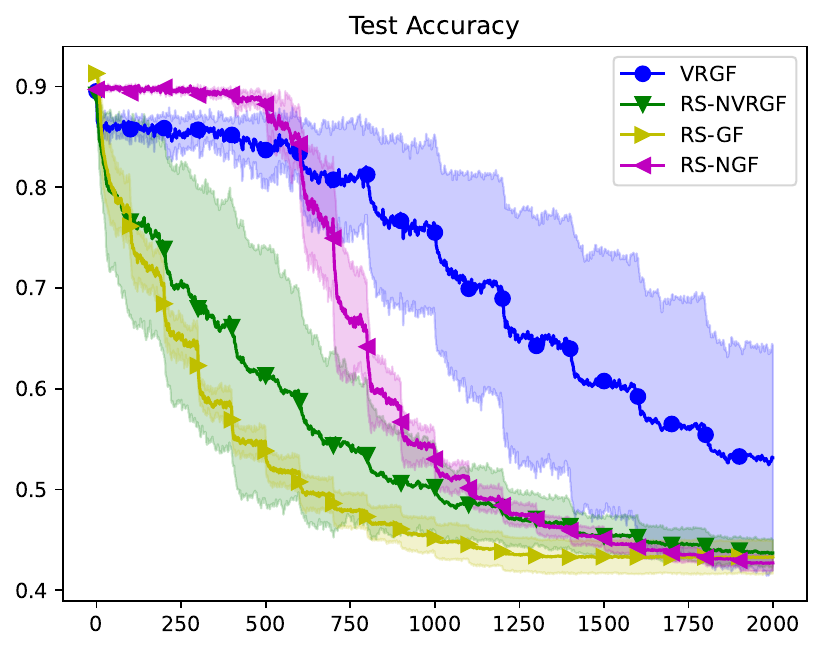}%
\end{minipage}
\end{centering}
\caption{\label{fig:cifar10}The attacking performance of the four algorithms is evaluated by plotting the test loss and test accuracy on the $y$-axis against the number of epochs on the $x$-axis.}
\end{figure*}

On the CIFAR10 dataset, RS-NGF achieves the lowest model accuracy after the attack, along with the smallest variance. In the early stages of the algorithm, RS-NGF progresses slowly due to the high noise in the stochastic gradients, which results in a large gradient norm and, consequently, a small step size. However, in the final 5,000 iterations, RS-NGF demonstrates more stable convergence behavior. In contrast, RS-GF rapidly increases the loss in the early stages but shows limited improvement in the later stages. Ultimately, all three algorithms—RS-NVRGF, RS-GF, and RS-NGF—converge to similar loss values, with the two normalized algorithms (RS-NVRGF and RS-NGF) exhibiting significant improvements even in the later stages.

\section{Conclusion and future work\label{sec:Future-work}}
This paper defines a new class of functions to address the restrictive global Lipschitz continuity assumption in nonsmooth optimization. We propose three algorithms---RS-GF, RS-NGF, and RS-NVRGF, that align with optimal convergence rates for globally Lipschitz problems. Specifically, while RS-GF and RS-NGF achieve $(\delta, \epsilon)$-Goldstein stationarity with $\tilde{\mathcal{O}}(\delta^{-1}\epsilon^{-4}d^{\frac{5}{2}})$ and $\tilde{\mathcal{O}}(\delta^{-1}\epsilon^{-4}d^{\frac{3}{2}})$ complexity, RS-NVRGF improves the rate to $\tilde{\mathcal{O}}(\delta^{-1}\epsilon^{-3}d^{\frac{3}{2}})$, matching the optimal $\epsilon$-bound under the global Lipschitz assumption, up to a logarithmic factor. 
In the theoretical analysis, our result relies on the assumption of knowing $\alpha$ and $\beta$.
A possible direction is to design a parameter-free method. 
\bibliography{ref}
\bibliographystyle{plainnat}

\newpage
\onecolumn
\appendix
\section{Appendix}

The appendix presents the missing proof in the paper.
The first part will provide more discussion about the generalized Lipschitz continuity and the subgradient growth condition.
The second part will show the missing proof in Section 2. The last part will show the formal version of the theorems and the missing proof in Section 3.

\section{Discussions about the definitions in Section 2: Generalized Lipschitz continuity}\label{sec:discussion_def}
Our definition begins from the growth condition of the norm of the subgradient. However, since our main contribution is to relax the global Lipschitz continuity assumption, a natural question is why don't we start from the definition of local Lipschitz continuity, such as the one in Lemma \ref{lem:lipschitz-1}? In addition, the setting of $\alpha$ and $\beta$ may be confusing. This prompts a related question: could we adopt a more flexible condition, such as $\abs{f(\xbf) - f(\ybf)} \leq L(\xbf, \ybf) \norm{\xbf - \ybf}$, where $L(\xbf, \ybf): \mathbb{R}^d \times \mathbb{R}^d \to \mathbb{R}_+$ replaces $\alpha$ and $\beta$, to achieve greater generality?

While such a formulation is indeed more general, it complicates the smoothness analysis of $f_\delta(x)$. The proof of Prop~\ref{prop:smoothness} Part 3) requires a growth condition on the subgradient norm of $f$ to derive the smoothness of $f_\delta(x)$. Using an arbitrary oracle $L(x, y)$ in place of $\alpha(x)$ would make it difficult. In contrast, involving $\alpha(x)$ and $\beta(x, r)$ directly supports our analysis within the randomized smoothing framework.

Moreover, to clarify the relationship between our definition and the local Lipschitz continuity setting, we provide a proof that local Lipschitz continuity implies a growth condition on the subgradient norm. We use a function of $x$ and $r$ where $r$ captures the distance between $x$ and $y$.

\begin{defn}
  The function $f:\Rbb^d\to\Rbb$ is $(\alpha,r)$ continuous if $|f(x)-f(y)|\leq\alpha(x)\|x-y\|$ holds for
all $y\in B_{r}(x)$, where $\alpha:\mathbb{R}^{d}\to\mathbb{R}_{+}$
is continuous.
\end{defn}
\begin{lem}
If $f$ is $(\alpha,r)$ continuous, then we have $\max\{\norm{\zeta}:\zeta\in\partial f(\xbf)\}\leq\alpha(\xbf)$ and for $\ybf$ such that $\norm{\ybf-\xbf}\leq r$, we have $\alpha(\xbf)-\alpha(\ybf)\leq \beta(\xbf,r)$ where $\beta(x,r)=\max_{u}\{|\|u\|-\alpha(x)|:u\in\partial f(y),y\in B_{r}(x)\}$.
\end{lem}

\begin{proof}
    Let $\Omega_{f}=\{x\in\mathbb{R}^{d} | f\text{ is not differentiable at point }x\}.$
For $x\in\Omega_{f}$ and vector $u$ such that $\|{u}\|=1$,
we have $|{f(x)-f(x+tu)}|\leq t\alpha(x).$ 
Then we know that 
\begin{equation*}
    \begin{aligned}\|{\nabla f(x)}\| & =\left|{\lim_{t\downarrow0}\frac{1}{t}(f(x+tu)-f(x))}\right|\\
 & =\lim_{t\downarrow0}\frac{1}{t}\left|{f(x+tu)-f(x)}\right|\leq\alpha(x).
\end{aligned}
\end{equation*}
Since $f$ is locally Lipschitz continuous, we have $$\partial f(x)=\text{conv}\{\xi\in\mathbb{R}^{d}|\exists x_{i}\subset\mathbb{R}^{d}\backslash\Omega_{f}\text{ such that }x_{i}\to x\text{ and }\nabla f(x_{i})\to\xi\}.$$
Since the norm is continuous in finite dimensions, then we have $\|{\xi}\|=\|{\lim_{i\to\infty}\nabla f(x_{i})}\|=\lim_{i\to\infty}\|{\nabla f(x_{i})}\|$.
Obviously, we have $x_{i}\to x$, $\alpha(x_{i})\to\alpha(x)$ since $\alpha$
is continuous function. It follows that 
\begin{equation*}
    \limsup_{i\to\infty}\|{\nabla f(x_{i})}\|\leq\limsup_{i\to\infty}\alpha(x_{i})=\lim_{i\to\infty}\alpha(x_{i})=\alpha(x).
\end{equation*}
Then, we have $\|{\xi}\|=\|{\lim_{i\to\infty}\nabla f(x_{i})}\|=\lim_{i\to\infty}\|{\nabla f(x_{i})}\|\leq\alpha(x)$.
Then by the convexity of norm, we complete the proof.

For the deviation in a neighborhood: define $\beta(x,r)=\max_{u}\{|\|u\|-\alpha(x)|:u\in\partial f(y),y\in B_{r}(x)\}$, where $\beta$ is non-decreasing in $r$.  
\end{proof}
\section{Proof and Properties missing in Section 3: Randomized smoothing}\label{sec:proof_for_def}
\subsection{Proof for Lemma~\ref{lem:lipschitz-1}}

Based on the Mean-value Theorem~\citep[Thm 2.3.7]{clarke1990optimization}, we have $\abs{f(\xbf)-f(\ybf)}\in \inner{\partial f(\xbf+\theta(\ybf-\xbf))}{\xbf-\ybf},$ for some $\theta\in(0,1)$
. Then it holds that
\begin{equation*}
    \begin{aligned}\abs{f(\xbf)-f(\ybf)} & \leq \alpha(\xbf+\theta(\ybf-\xbf))\norm{\xbf-\ybf}\\
 & \aleq (\alpha(\xbf)+\beta(\xbf,\theta\norm{\ybf-\xbf}))\norm{\xbf-\ybf}\\
 & \bleq(\alpha(\xbf)+\beta(\xbf,\norm{\ybf-\xbf}))\norm{\xbf-\ybf},
\end{aligned}
\end{equation*}
where $(a)$ comes from \eqref{eq:lipschitz_alpha} and $\alpha(\xbf+\theta(\ybf-\xbf))\leq\alpha(\xbf)+\beta(\xbf,\theta\norm{\ybf-\xbf})$, and $(b)$ comes from that $\beta(\xbf,r)$ is non-decreasing in $r$.

\subsection{Proof of Example~\ref{ex:non_decreasing}}

Since $\sup_{\zeta\in\partial f(\xbf)}\|\zeta\|\leq\gamma(\norm{\xbf})$
where $\gamma$ is non-decreasing on $\norm{\xbf}$ and is finite
on any bounded set of $\norm{\xbf}$, then 
we have 
\begin{equation*}
    \abs{\gamma(\norm{\xbf})-\gamma(\norm{\ybf})}\leq\gamma\norm{\xbf})+\gamma(\norm{\xbf-\ybf}+\norm{\xbf}).
\end{equation*}

\subsection{Proof of Example \ref{ex:add}}

Consider the composite function $f(x) = g(\xbf) + h(\xbf)$. From Section 3.2 in \cite{bagirov2014introduction}, we know that the subdifferential of $f$ satisfies:
\[
\partial f(\xbf) \subseteq \partial g(\xbf) + \partial h(\xbf).
\]
Thus, the upper bound of the subgradient norm of $f$ is bounded by:
\[
\sup_{\zeta \in \partial f(\xbf)} \|\zeta\| \leq \sup_{\zeta \in \partial g(\xbf) + \partial h(\xbf)} \|\zeta\| \leq \alpha_h(\xbf) + \alpha_g(\xbf).
\]
Next, we analyse the difference between the Lipschitz constants of points $\xbf$ and $\ybf$ by the definition of  $(\alpha,\beta)$ subgradient growth condition and the triangle inequality:
\[
|\alpha_h(\xbf) + \alpha_g(\xbf) - \alpha_h(\ybf) - \alpha_g(\ybf)| \leq \beta_g(\xbf, \|\xbf - \ybf\|) + \beta_h(\xbf, \|\xbf - \ybf\|).
\]

\subsection{Proof of Lemma~\ref{lem:final}}
    Since $\alpha(\xbf)$ is a continuous function, then $\alpha$ is bounded in any bounded set. For any $\xbf,\ybf\in\Rbb^{d}$, $\sup_{\theta\in(0,1)}\alpha(\xbf+\theta(\ybf-\xbf))$
is also bounded. Then the integral and derivative of $f(\xbf+\delta\wbf)$  can be exchanged
by the dominated convergence theorem and the proof of Theorem 3.1
in \cite{lin2022gradient} holds. The rest of the proof is entirely similar to the one in \cite{lin2022gradient}.

\subsection{Proof of Example~\ref{ex:composite}}

Consider the composite function $f(x) = g \circ \hbf(\xbf)$. According to the chain rule of Clarke subdifferential (\citep[section~3.2]{bagirov2014introduction}), we have
\[
\partial f(\xbf) \subseteq \text{conv} \left\{ \partial \hbf(\xbf)^\top \partial g(\hbf(\xbf)) \right\}.
\]
Thus, the upper bound of the subgradient norm of $f$ satisfies:
\[
\sup_{\zeta \in \partial f(\xbf)} \|\zeta\| \leq \sup_{\zeta \in \text{conv} \left\{ \partial \hbf(\xbf)^\top \partial g(\hbf(\xbf)) \right\}} \|\zeta\| \leq \sup_{\zeta \in \partial \hbf(\xbf)^\top \partial g(\hbf(\xbf))} \|\zeta\| \leq \alpha_h(\xbf) \alpha_g(\hbf(\xbf)).
\]
Next, we examine the variation between the subgradient upper bound at $\xbf$ and $\ybf$:
\[
\begin{aligned}
& \abs{\alpha_{\hbf}(\xbf)\alpha_{g}(\hbf(\xbf))-\alpha_{\hbf}(\ybf)\alpha_{g}(\hbf(\ybf))} \\
& \leq\abs{\alpha_{\hbf}(\xbf)\alpha_{g}(\hbf(\xbf))-\alpha_{\hbf}(\xbf)\alpha_{g}(\hbf(\ybf))+\alpha_{\hbf}(\xbf)\alpha_{g}(\hbf(\ybf))-\alpha_{\hbf}(\ybf)\alpha_{g}(\hbf(\ybf))}\\
 & \leq\abs{\alpha_{\hbf}(\xbf)\alpha_{g}(\hbf(\xbf))-\alpha_{\hbf}(\xbf)\alpha_{g}(\hbf(\ybf))}+\abs{\alpha_{\hbf}(\xbf)\alpha_{g}(\hbf(\ybf))-\alpha_{\hbf}(\ybf)\alpha_{g}(\hbf(\ybf))}\\
 & \aleq\alpha_{\hbf}(\xbf)\beta_{g}(\hbf(\xbf),\norm{\hbf(\xbf)-\hbf(\ybf)})+\beta_{\hbf}(\xbf,\norm{\xbf-\ybf})(\alpha_{g}(\hbf(\xbf))+\beta_{g}(\hbf(\xbf),\norm{\hbf(\xbf)-\hbf(\ybf)}))\\
 & =(\alpha_{\hbf}(\xbf)+\beta_{\hbf}(\xbf,\norm{\xbf-\ybf}))\beta_{g}(\hbf(\xbf),\norm{\hbf(\xbf)-\hbf(\ybf)})+\beta_{\hbf}(\xbf,\norm{\xbf-\ybf})\alpha_{g}(\hbf(\xbf))\\
 & \bleq(\alpha_{\hbf}(\xbf)+\beta_{\hbf}(\xbf,\norm{\xbf-\ybf}))\beta_{g}(\hbf(\xbf),(\alpha_{\hbf}(\xbf)+\beta_{\hbf}(\xbf,\norm{\ybf-\xbf}))\norm{\xbf-\ybf})\\
 & +\beta_{\hbf}(\xbf,\norm{\xbf-\ybf})\alpha_{g}(\hbf(\xbf)),
\end{aligned}
\]
where $(a)$ and $(b)$ comes from the subgradient growth condition.
This completes the proof.

\subsection{Proof of Proposition~\ref{prop:smoothness}}

Part 1. We show $\abs{f_{\delta}(\xbf)-f(\xbf)}\leq\delta(\alpha(\xbf)+\beta(\xbf,\delta))$. 
With the definition of ($\alpha$, $\beta$) subgradient growth condition, we
have
\[
\begin{aligned}\abs{f(\xbf)-f_{\delta}(\xbf)} & =\abs{\Expe_{\wbf}(f(\xbf)-f(\xbf+\delta\wbf))}\\
 & \leq\delta(\alpha(\xbf)+\beta(\xbf,\delta)).
\end{aligned}
\]

Part 2.
We first show the subgradient upper bound.
\[\begin{aligned}
    \norm{\nabla f_\delta(\mathbf{x})}&=\norm{\Expe_\wbf [\nabla f(\mathbf{x}+\delta\wbf)]}\
    &\leq \alpha(\mathbf{x})+\beta(\mathbf{x},\delta).
\end{aligned}\]
For the second property:
\[\begin{aligned}
  \norm{\nabla f_\delta(\ybf)}&=\norm{\Expe_\wbf [\nabla f(\ybf+\delta\wbf)]}\\
    &\leq \alpha(\mathbf{x})+\Expe_\wbf[\beta(\mathbf{x},\norm{\ybf-\mathbf{x}+\delta\wbf})]\\
    &\leq \alpha(\mathbf{x})+\beta(\mathbf{x},\norm{\ybf-\mathbf{x}}+\delta).
\end{aligned}\]
Then we finish the proof.

Now We show $\abs{f_{\delta}(\xbf)-f_{\delta}(\ybf)}\leq\alpha(\norm{\xbf}+\norm{\xbf-\ybf}+\delta)\norm{\xbf-\ybf}$.
By definition of $f_\delta(\cdot)$, we have
\[
\begin{aligned}\abs{f_{\delta}(\xbf)-f_{\delta}(\ybf)} & =\abs{\Expe_{\wbf}(f(\xbf+\delta\wbf)-f(\ybf+\delta\wbf))}\\
 & \leq\Expe_{\wbf}(\sup_{\theta\in(0,1)}\alpha(\xbf+\delta\wbf+\theta(\ybf-\xbf)))\norm{\xbf-\ybf}\\
 & \aleq(\alpha(\xbf)+\beta(\xbf,\norm{\xbf-\ybf}+\delta))\norm{\xbf-\ybf},
\end{aligned}
\]
where $(a)$ comes from $\alpha(\xbf+\delta\wbf+\theta(\ybf-\xbf))\leq\alpha(\xbf)+\beta(\xbf,\norm{\delta\wbf+\theta(\ybf-\xbf)})\leq\alpha(\xbf)+\beta(\xbf,\norm{\xbf-\ybf}+\delta).$

Part 3. We show $\norm{\nabla f_{\delta}(\xbf)-\nabla f_{\delta}(\ybf)}\leq\frac{\sqrt{d}\left(2\alpha(\xbf)+\beta(\xbf,\delta)+\beta(\xbf,\norm{\xbf-\ybf}+\delta)\right)}{2\delta}\norm{\xbf-\ybf}$.
Since $f$ is locally Lipschitz continuous, $f$ is almost everywhere differentiable by Rademacher's theorem. Then we have $\nabla f_{\delta}(\xbf)=\Expe_{\wbf}[\nabla f(\xbf+\delta\wbf)]$.
Then we have 
\[
\begin{aligned}\nabla f_{\delta}(\xbf)-\nabla f_{\delta}(\ybf) & =\Expe[\nabla f(\xbf+\delta\wbf)]-\Expe[\nabla f(\ybf+\delta\wbf)]\\
 & =\frac{1}{\Vol(\Bbb_{1}(\zerobf))}\left(\int_{u\in\Bbb_{1}(\zerobf)}\nabla f(\xbf+\delta\wbf)d\wbf-\int_{u\in\Bbb_{1}(\zerobf)}\nabla f(\ybf+\delta\wbf)d\wbf\right)\\
 & =\frac{1}{\Vol(\Bbb_{\delta}(\zerobf))}\left(\int_{\zbf\in\Bbb_{\delta}(\xbf)}\nabla f(\zbf)d\zbf-\int_{\zbf\in\Bbb_{\delta}(\ybf)}\nabla f(\zbf)d\zbf\right).
\end{aligned}
\]

Case 1 : $\norm{\xbf-\ybf}\geq2\delta.$ Then we have 
\[
\begin{aligned}\norm{\nabla f_{\delta}(\xbf)-\nabla f_{\delta}(\ybf)} & \leq2\alpha(\xbf)+\beta(\xbf,\delta)+\beta(\xbf,\norm{\xbf-\ybf}+\delta)\\
 & \leq\frac{\left(2\alpha(\xbf)+\beta(\xbf,\delta)+\beta(\xbf,\norm{\xbf-\ybf}+\delta)\right)}{2\delta}\norm{\ybf-\xbf}\\
 & \leq\frac{c\sqrt{d}\left(2\alpha(\xbf)+\beta(\xbf,\delta)+\beta(\xbf,\norm{\xbf-\ybf}+\delta)\right)}{2\delta}\norm{\ybf-\xbf}.
\end{aligned}
\]

Case 2: $\norm{\xbf-\ybf}\leq2\delta$. Then we have 
\[
\begin{aligned}
& \norm{\nabla f_{\delta}(\xbf)-\nabla f_{\delta}(\ybf)} \\
& =\frac{1}{\Vol(\Bbb_{\delta}(\zerobf))}\left\Vert \int_{\zbf\in\Bbb_{\delta}(\xbf)\backslash\Bbb_{\delta}(\ybf)}\nabla f(\zbf)d\zbf-\int_{\zbf\in\Bbb_{\delta}(\ybf)\backslash\Bbb_{\delta}(\xbf)}\nabla f(\zbf)d\zbf\right\Vert \\
 & \leq\frac{1}{\Vol(\Bbb_{\delta}(\zerobf))}\left(\left\Vert \int_{\zbf\in\Bbb_{\delta}(\xbf)\backslash\Bbb_{\delta}(\ybf)}\nabla f(\zbf)d\zbf\right\Vert +\left\Vert \int_{\zbf\in\Bbb_{\delta}(\ybf)\backslash\Bbb_{\delta}(\xbf)}\nabla f(\zbf)d\zbf\right\Vert \right)\\
 & \leq\frac{1}{\Vol(\Bbb_{\delta}(\zerobf))}\left(\int_{\zbf\in\Bbb_{\delta}(\xbf)\backslash\Bbb_{\delta}(\ybf)}\norm{\nabla f(\zbf)}d\zbf+\int_{\zbf\in\Bbb_{\delta}(\ybf)\backslash\Bbb_{\delta}(\xbf)}\norm{\nabla f(\zbf)}d\zbf\right)\\
  & \aleq\frac{1}{\Vol(B_{\delta}(\zerobf))}\left((\alpha(\xbf)+\beta(\xbf,\delta))\Vol(\Bbb_{\delta}(\xbf)\backslash\Bbb_{\delta}(\ybf))+(\alpha(\xbf)+\beta(\xbf,\norm{\xbf-\ybf}+\delta))\Vol(\Bbb_{\delta}(\ybf)\backslash\Bbb_{\delta}(\xbf))\right)\\
 & \bleq\frac{c\sqrt{d}(2\alpha(\xbf)+\beta(\xbf,\delta)+\beta(\xbf,\norm{\xbf-\ybf}+\delta))}{2\delta}\norm{\xbf-\ybf},
\end{aligned}
\]
where $(a)$ comes from the definition of generalized Lipschitz continuity
and $(b)$ is from the proof of Proposition 2.2 in \cite{lin2022gradient}.

\subsection{Proof of Lemma~\ref{lem:smoothness_descent}}

By the local Lipschitz continuity, we have
\[
\begin{aligned}f_{\delta}(\ybf)-f_{\delta}(\xbf) & =\int_{0}^{1}\inner{\nabla f_{\delta}(\xbf+\theta(\ybf-\xbf))}{\ybf-\xbf}d\theta\\
 & =\int_{0}^{1}\inner{\nabla f_{\delta}(\xbf+\theta(\ybf-\xbf))-\nabla f_{\delta}(\xbf)+\nabla f_{\delta}(\xbf)}{\ybf-\xbf}d\theta\\
 & =\inner{\nabla f_{\delta}(\xbf)}{\ybf-\xbf}+\int_{0}^{1}\inner{\nabla f_{\delta}(\xbf+\theta(\ybf-\xbf))-\nabla f_{\delta}(\xbf)}{\ybf-\xbf}d\theta\\
 & \aleq\inner{\nabla f_{\delta}(\xbf)}{\ybf-\xbf}+\int_{0}^{1}\norm{\nabla f_{\delta}(\xbf+\theta(\ybf-\xbf))-\nabla f_{\delta}(\xbf)}\norm{\ybf-\xbf}d\theta\\
 & \bleq\inner{\nabla f_{\delta}(\xbf)}{\ybf-\xbf}+\int_{0}^{1}\theta\frac{c\sqrt{d}(2\alpha(\xbf)+\beta(\xbf,\delta)+\beta(\xbf,\theta\norm{\xbf-\ybf}+\delta))}{2\delta}\norm{\ybf-\xbf}^{2}d\theta\\
 & \cleq\inner{\nabla f_{\delta}(\xbf)}{\ybf-\xbf}+\int_{0}^{1}\theta\frac{c\sqrt{d}(2\alpha(\xbf)+\beta(\xbf,\delta)+\beta(\xbf,\norm{\xbf-\ybf}+\delta))}{2\delta}\norm{\ybf-\xbf}^{2}d\theta\\
 & =\inner{\nabla f_{\delta}(\xbf)}{\ybf-\xbf}+\frac{\ell(\xbf,r)}{2}\norm{\ybf-\xbf}^{2}
\end{aligned}
\]
where $(a)$ comes from Cauchy-Schwartz inequality, $(b)$ comes from the
generalized smoothness of $f_{\delta}$ and $(c)$ comes from that $\beta(\xbf,r)$
is non-decreasing in $r$.

\subsection{Properties about the gradient estimator}\label{subsec:gradient_estimator_property}
Below, we establish some properties of the gradient estimators. 
The unbiasedness of $\gbf(\xbf)$ and $\gbf(\xbf,S_B)$ is a direct consequence of prior results~\citep[Lemma D.1]{lin2022gradient}, which is stated here for completeness.
\begin{lem}[Lemma D.1\citep{lin2022gradient}]
\label{lem:unbias}
For any $\xbf$
and $\gbf(\xbf,\wbf)$ ,
where $\wbf$ is sampled from  $\Ubf(\mathbb{S}^{d-1})$, we have
$\Expe_{\wbf}[\gbf(\xbf,\wbf)]=\nabla f_{\delta}(\xbf)$. 
\end{lem}
Next, we compute the second-order moment of the gradient estimators.
\begin{lem}
\label{lem:variance}For any $\xbf\in \Rbb^d$ and $\gbf(\xbf,\wbf)$ calculated by \eqref{eq:grad_esti_batch} with $B=1$
where $\wbf$ is sampled from $\Ubf(\mathbb{S}^{d-1})$, we have 
\[
\Expe[\norm{\gbf(\xbf,\wbf)}^{2}]\leq\sigma^2(\xbf),
\]where $\sigma(\xbf):=\sqrt{16\sqrt{2\pi}d((\alpha(\xbf)+\beta(\xbf,2\delta)))^{2}}$.
Furthermore, when using a minibatch $S_B=\{\wbf_{i}\}_{i=1}^{B}$
sampled i.i.d. from $\Ubf(\mathbb{S}^{d-1})$  and the gradient is estimated as \eqref{eq:grad_esti_batch},
 we have 
\begin{equation*}
\Expe[\norm{\gbf(\xbf,S_B)}^{2}]\leq\frac{\sigma^2(\xbf)}{B}.
\end{equation*}
\end{lem}
In view of the definition of $\gbf(\xbf,\wbf)$ and inequality (D.1) in \cite{lin2022gradient},
we have
\begin{equation}\label{eq:moment-bound}
\Expe[\norm{\gbf(\xbf,\wbf)}^{2}]\leq\frac{d^{2}}{\delta^{2}}\Expe[(f(\xbf+\delta\wbf)-\Expe[f(\xbf+\delta\wbf)])^{2}].
\end{equation}
Let $h(\wbf)=f(\xbf+\delta\wbf)$. Since $f$ has $(\alpha$,$\beta)$
subgradient growth condition, for any $\wbf_{1}$ and $\wbf_{2}$ , we have
\[
\begin{aligned}\abs{h(\wbf_{1})-h(\wbf_{2})} & =\abs{f(\xbf+\delta\wbf_{1})-f(\xbf+\delta\wbf_{2})}\\
 & \leq\sup_{\theta\in(0,1)}\alpha(\xbf+\theta\delta(\wbf_{1}-\wbf_{2}))\norm{\wbf_{1}-\wbf_{2}}\\
 & \leq(\alpha(\xbf)+\beta(\xbf,2\delta))\norm{\wbf_{1}-\wbf_{2}}
\end{aligned}
\]
where the second inequality uses the definition of $(\alpha,\beta)$ subgradient growth condition.
Then, by  Proposition 3.11 and Example 3.12 in \cite{wainwright2019high},
for any $\theta\geq0$, we have
\[
\Pbb(\abs{h(\wbf)-\Expe[h(\wbf)]}\ge\theta)\leq2\sqrt{2\pi}e^{-\frac{\theta^{2}d}{8\delta^{2}((\alpha(\xbf)+\beta(\xbf,2\delta)))^{2}}}.
\]
It follows that
\[
\begin{aligned}\Expe[(h(\wbf)-\Expe h(\wbf))^{2}] & =\int_{0}^{\infty}\Pbb((h(\wbf)-\Expe[h(\wbf)])^{2}\ge\theta)\rm{d}\theta\\
 & =\int_{0}^{\infty}\Pbb(\abs{h(\wbf)-\Expe[h(\wbf)]}\ge\sqrt{\theta})\rm{d}\theta\\
 & \leq\int_{0}^{\infty}2\sqrt{2\pi}e^{-\frac{\theta d}{8\delta^{2}((\alpha(\xbf)+\beta(\xbf,2\delta)))^{2}}}\rm{d}\theta\\
 & =\frac{16\sqrt{2\pi}\delta^{2}(\alpha(\xbf)+\beta(\xbf,2\delta))^{2}}{d}.
\end{aligned}
\]
By the definition of $h$, we have 
\[
\Expe[(f(\xbf+\delta\wbf)-\Expe[f(\xbf+\delta\wbf)])^{2}]\leq\frac{16\sqrt{2\pi}\delta^{2}((\alpha(\xbf)+\beta(\xbf,2\delta)))^{2}}{d}.
\]
Combining the above result and \eqref{eq:moment-bound} gives
\[
\Expe[\norm{\gbf(\xbf,\wbf)}^{2}]\leq16\sqrt{2\pi}d((\alpha(\xbf)+\beta(\xbf,2\delta)))^{2}.
\]
For the mini-batch gradient estimator~\eqref{eq:grad_esti_batch},
we similarly have
\[
\begin{aligned}\Expe\left[\norm{\gbf(\xbf,S_B)}^{2}\right] & =\frac{1}{B^{2}}\Expe\left[\Bnorm{\sum_{i=1}^{B}\frac{d(F(\xbf+\delta\wbf_{i})-F(\xbf-\delta\wbf_{i}))}{2\delta}\wbf_{i}}^{2}\right]\\
 & \leq\frac{1}{B^{2}}\sum_{i=1}^{B}\Expe\left[\Bnorm{\frac{d(F(\xbf+\delta\wbf_{i})-F(\xbf-\delta\wbf_{i}))}{2\delta}\wbf_{i}}^{2}\right]\\
 & \leq\frac{16\sqrt{2\pi}d((\alpha(\xbf)+\beta(\xbf,2\delta)))^{2}}{B}.
\end{aligned}
.
\]
\subsection{Proof of Lemma~\ref{lem:bounded_estimator}}

Based on the Mean-value Theorem~\citep[Thm 2.3.7]{clarke1990optimization}, we have ${f(\xbf)-f(\ybf)}\in\{\inner{\zeta}{\xbf-\ybf}:\zeta\in\partial f(\xbf+\theta_0(\ybf-\xbf)),\}$, for some $\theta_0\in(0,1)$. Applying Cauchy's inequality, it holds that
\[
\begin{aligned}
\abs{f(\xbf+\delta\wbf)-f(\xbf-\delta\wbf)} & 
\le \sup\{\norm{\zeta}\cdot\norm{\xbf-\ybf}: \zeta\in\partial f(\xbf+\theta_0(\ybf-\xbf))\}\\
&\leq\sup_{\theta\in(0,1)}\alpha(\xbf+\delta\wbf+2\delta\theta\wbf)\cdot\norm{2\delta\wbf}\\
 & \aleq\sup_{\theta\in(0,1)}(\alpha(\xbf)+\beta(\xbf,\delta+2\delta\theta))\norm{2\delta\wbf}\\
 & \bleq2\delta(\alpha(\xbf)+\beta(\xbf,3\delta)),
\end{aligned}
\]
where $(a)$ comes from the definition of $(\alpha,\beta)$ subgradient growth condition and $\norm{\wbf}=1$ and $(b)$ comes from that $\beta(\xbf,r)$ is non-decreasing on $r$.
Then we have 
\[
\norm{\gbf(\xbf,\wbf)}=\Big\|\frac{d(f(\xbf+\delta\wbf)-f(\xbf-\delta\wbf))}{2\delta}\wbf\Big\| \leq d(\alpha(\xbf)+\beta(\xbf,3\delta)).
\]
\section{Proof missing in Section 4: Algorithms and convergence analysis}\label{sec:proof_for_alg}
\subsection{Discussion about Assumption~\ref{assu:level_set}}\label{subsec:discuss_assum}
We introduce Assumption~\ref{assu:level_set} that the \( f_\delta \) has bounded level set. We now present sufficient conditions for this assumption.
\begin{lem}
    Let \( f \) be a generalized smooth function with parameters \( (\alpha, \beta) \) and $ Q=\{\xbf:f_\delta(\xbf)\leq\nu\}$. If either of the following conditions is satisfied:
    \begin{enumerate}
        \item  For any \( \delta>0 \), the function \( \frac{f(\xbf)}{\alpha(\xbf) + \beta(\xbf, \delta)} \) is coercive, namely, 
\begin{equation}\label{eq:corecive}
            \lim_{\norm{\xbf} \to \infty} \frac{f(\xbf)}{\alpha(\xbf) + \beta(\xbf, \delta)} = \infty.
\end{equation}
        \item  Suppose the optimal solution set of $f$ $S$ is bounded. For any $\xbf\in\Rbb^d$ and $\xbf^*\in S$, there exist $\gamma>0,p\geq1$ such that
        \[
        f(\xbf) - f^* \geq \sigma \norm{\xbf - \operatorname{proj}_{S}(\xbf)}^p,
        \]

    \end{enumerate}
    then \( f_\delta \) satisfies Assumption~\ref{assu:level_set}, namely, $Q$ is bounded.
\end{lem}
\begin{proof}
 
We first show the sufficiency of the first condition. From \eqref{eq:corecive}, it holds that $\{\xbf:f(\xbf)\leq \delta(\alpha(\xbf)+\beta(\xbf,\delta))+\nu\}$  is bounded.
From Lemma~\ref{lem:lipschitz-1}, we have \[
f(\xbf)-f_\delta(\xbf)\leq \delta(\alpha(\xbf)+\beta(\xbf,\delta)).
\]
Then  $Q$ is bounded.

We now show the proof of the second condition. For any $\wbf\in \Ubf(\Bbb_1(\zerobf))$, we have
\[\begin{aligned}
    f_\delta(\xbf)&=\Expe_\wbf f(\xbf+\delta\wbf)\\
    &\geq f^{\star}+\sigma\Expe_\wbf[\norm{\xbf+\delta\wbf-\operatorname{proj}_{S}(\xbf+\delta\wbf)}]\\
    &\ageq f^{\star}+\sigma\norm{\xbf-\operatorname{proj}_{S}(\xbf)},\\
\end{aligned}\]
where $(a)$ comes from the convexity of $\norm{\cdot}$.
For any $\xbf\in Q$ we have $\norm{\xbf-\operatorname{proj}_{S}(\xbf)}^p\leq\frac{1}{\sigma}(\nu-f^*)$. Since  the set $S$ is bounded, then $Q$ is bounded.
\end{proof}
\subsection{Technique Lemmas}\label{sec:tech}

We first present Bernstein's Lemma, a concentration result frequently used in our subsequent proof. We adopt the variant in  \citet[Lemma 2]{chezhegov2024gradient}.
\begin{lem}
\label{lem:bernstein_lemma}Let the sequence of random variables $\{X_{i}\}_{i\geq1}$
form a martingale difference sequence, i.e., $\Expe\left[X_{i}\ |\ X_{i-1},\ldots,X_{1}\right]=0$
for all $i\geq1$. Assume that conditional variances $\sigma_{i}^{2}=\Expe\left[X_{i}^{2}\ |\ X_{i-1},\ldots,X_{1}\right]$
exist and are bounded, and also assume that there exists a deterministic
constant $c>0$ such that $|X_{i}|\leq c$ almost surely for all $i\geq1$.
Then for all $b>0$, $G>0$ and $n\geq1$, the following inequality
holds:
\[
\mathbb{P}\left\{ \abs{\sum\limits _{i=1}^{n}X_{i}}>b\text{ and }\sum\limits _{i=1}^{n}\sigma_{i}^{2}\leq G\right\} \leq2\exp\left(-\frac{b^{2}}{2G+\frac{2cb}{3}}\right).
\]
\end{lem}
Next, we show that the gradient estimator is uniformly bounded above, which is important for the subsequent convergence analysis.
\begin{lem}\label{lem:bounded_estimator}
    For any $\xbf\in \Rbb^d$ and $\gbf(\xbf,\wbf)$ 
where $\wbf\in \mathbb{S}^{d-1}$, we have 
\[
\norm{\gbf(\xbf,\wbf)}\leq d(\alpha(\xbf)+\beta(\xbf,3\delta)).
\]
\end{lem}
Here is a technique lemma that occurs as Lemma 7 in \cite{hubler2024gradient}.
We will use it in the proof about the RS-NGF and RS-NVRGF.
\begin{lem}\label{lem:huber}
(Lemma 7 in \cite{hubler2024gradient}) For all $\xbf,\ybf\in\Rbb^{d}$
with $\ybf\neq\zerobf$, we have 
\[
\frac{\inner{\xbf}{\ybf}}{\norm{\ybf}}\geq\norm{\xbf}-2\norm{\xbf-\ybf}.
\]
\end{lem}

The following lemma is a natural extension of the bound of the variance
of the estimates $\mbf_{t}$ in expectation. Similar proof can be
found in \cite{chen2023faster,reisizadeh2023variance}. 
\begin{lem}
\label{lem:sum_theta_b}Assume $f$ has $(\alpha,\beta)$ subgradient growth condition. Then the iterates generated by Algorithm \ref{alg:storm}
with parameter $\{\eta_{s}\}_{s=1}^{t}$ satisfy
\[
\begin{aligned}\Expe\left[\norm{\nabla f_{\delta}(\xbf_{t})-\mbf_{t}}^{2}\right] & \leq\frac{\sigma^{2}(\xbf_{t_{0}})}{B_{t_{0}}}+\sum_{s=t_{0}+1}^{t}\frac{\eta_{s}^{2}}{b_{s}}\frac{d^{2}(\alpha(\xbf_{s-1})+\beta(\xbf_{s-1},\delta+\eta_{s}))^{2}}{\delta^{2}}\end{aligned}
,
\]
where $t_{0}$ is the most recent iterate to $t$ which $t_{0}\bmod q=0$.
\end{lem}
\begin{proof}
From Lemma \ref{lem:lipschitz-1}, we have
\begin{equation}
    \begin{aligned}
    & \norm{\gbf(\xbf,\wbf)-\gbf(\ybf,\wbf)}^{2} \\
    & \aleq\frac{d^{2}(f(\xbf+\delta\wbf)-f(\ybf+\delta\wbf))^{2}}{2\delta^{2}}+\frac{d^{2}(f(\xbf-\delta\wbf)-f(\ybf-\delta\wbf))^{2}}{2\delta^{2}}\\
 & \bleq\frac{d^{2}\norm{\xbf-\ybf}^{2}}{2\delta^{2}}\left((\sup_{\theta\in(0,1)}\alpha(\xbf+\delta\wbf+\theta(\ybf-\xbf)))^{2}+(\sup_{\theta\in(0,1)}\alpha(\xbf-\delta\wbf+\theta(\ybf-\xbf)))^{2}\right)\\
 & \cleq\frac{d^{2}(2\alpha(\xbf)+\beta(\xbf,\delta+\norm{\xbf-\ybf}))^{2}}{\delta^{2}}\norm{\xbf-\ybf}^{2},
\end{aligned}\label{eq:almost_sure_smooth-1}
\end{equation}
where $(a)$ comes from the Cauchy-Schwartz inequalities, $(b)$ comes from
Lemma \ref{lem:lipschitz-1} and $(c)$ comes from the definition
of $(\alpha,\beta)$ Lipschitz continuity.

We can also notice that for any sample set $S$, we have
\begin{equation}
\Expe\inner{\mbf_{t-1}-\nabla f_{\delta}(\xbf_{t-1})}{\gbf(\xbf_{t},S)-\nabla f_{\delta}(\xbf_{t})}=0,\label{eq:zero_inner_product1-1}
\end{equation}
 and 
\begin{equation}
\Expe\inner{\mbf_{t-1}-\nabla f_{\delta}(\xbf_{t-1})}{\gbf(\xbf_{t},S)-\nabla f_{\delta}(\xbf_{t}))-(\gbf(\xbf_{t-1},S)-\nabla f_{\delta}(\xbf_{t-1}))}=0.\label{eq:zero_inner_product_2-1}
\end{equation}

For any $t$ satisfying $t\bmod q=0$, we have 
\[
\begin{aligned}\Expe\left[\norm{\mbf_{t}-\nabla f_{\delta}(\xbf_{t})}^{2}\right] \leq\frac{\sigma^{2}(\xbf_{t})}{B}.\end{aligned}
\]

We denote $t_{0}$ as the most recent iterate to $t$ such that $t_{0}\bmod q=0$.
Then for any $t_{0}\leq t\leq t_{0}+q$, we have
\[
\begin{aligned}
& \Expe\left[\norm{\mbf_{t}-\nabla f_{\delta}(\xbf_{t})}^{2}\right]\\
& =\Expe\left[\norm{\mbf_{t-1}-\nabla f_{\delta}(\xbf_{t-1})-\gbf(\xbf_{t-1},S_{b_{t-1}})+\gbf(\xbf_{t},S_{b_{t-1}})+\nabla f_{\delta}(\xbf_{t-1})-\nabla f_{\delta}(\xbf_{t})}^{2}\right]\\
 & \aeq\Expe\left[\norm{\mbf_{t-1}-\nabla f_{\delta}(\xbf_{t-1})}^{2}+\norm{\gbf(\xbf_{t-1},S_{b_{t-1}})+\gbf(\xbf_{t},S_{b_{t-1}})+\nabla f_{\delta}(\xbf_{t-1})-\nabla f_{\delta}(\xbf_{t})}^{2}\right]\\
 & \bleq\Expe\left[\norm{\mbf_{t-1}-\nabla f_{\delta}(\xbf_{t-1})}^{2}\right]+\frac{1}{b_{t-1}}\Expe\left[\norm{\gbf(\xbf_{t-1},\wbf)+\gbf(\xbf_{t},\wbf)+\nabla f_{\delta}(\xbf_{t-1})-\nabla f_{\delta}(\xbf_{t})}^{2}\right]\\
 & \cleq\Expe\left[\norm{\mbf_{t-1}-\nabla f_{\delta}(\xbf_{t-1})}^{2}\right]+\frac{\eta_{t-1}^{2}}{b_{t-1}}\frac{d^{2}(\alpha(\xbf_{t-1})+\beta(\xbf_{t-1},\delta+\eta_{t-1}))^{2}}{\delta^{2}}\\
 & \leq\Expe\left[\norm{\mbf_{t_{0}}-\nabla f_{\delta}(\xbf_{t_{0}})}^{2}\right]+\sum_{s=t_{0}+1}^{t}\frac{\eta_{s}^{2}}{b_{s}}\frac{d^{2}(\alpha(\xbf_{s})+\beta(\xbf_{s},\delta+\eta_{s}))^{2}}{\delta^{2}}\\
 & \dleq\frac{\sigma^{2}(\xbf_{t_{0}})}{B_{t_{0}}}+\sum_{s=t_{0}+1}^{t}\frac{\eta_{s}^{2}}{b_{s}}\frac{d^{2}(\alpha(\xbf_{s})+\beta(\xbf_{s},\delta+\eta_{s}))^{2}}{\delta^{2}},
\end{aligned}
\]
where $(a)$ comes from \eqref{eq:zero_inner_product1-1}, $(b)$ comes from the fact that for any independent and identically distributed sequence $\{X_i\}_{i=1}^{n}$ with zero mean, $\Expe[\norm{\sum_{i=1}^{n}X_i}^2]=n\Expe[\norm{X_1}^2]$, $(c)$ comes from \eqref{eq:almost_sure_smooth-1}, and $(d)$ comes from Lemma~\ref{lem:variance} and \eqref{eq:zero_inner_product_2-1}.
\end{proof}
\subsection{Convergence rate of Algorithm~\ref{alg:sgd}}\label{sec:sgd}

We formally state Theorem~\ref{thm:result_sgd}, which shows the main convergence property of Algorithm~\ref{alg:sgd}.

\medskip

\textbf{Theorem.~\ref{thm:result_sgd}} Under Assumption~\ref{assu:level_set}, let the stepsize
$\eta_{t} =\frac{1}{\log\left(\frac{2T}{p}\right)\sqrt{T}}C_{t},$ where $p\in(0,1)$,
$$
\begin{aligned}
    C_{t} &= \min\left\{ \frac{\Delta }{2(d+1)(\alpha(\xbf_{t})+\beta(\xbf_{t},3\delta)+1)^{2}},\right.\\
& \left.\frac{\sqrt{6}}{12}\frac{\Delta }{\sigma(\xbf_{t})(\alpha(\xbf_{t})+\beta(\xbf_{t},\delta))},\frac{1}{d(\alpha(\xbf_{t})+\beta(\xbf_{t},3\delta))}\sqrt{\frac{\Delta }{\ell(\xbf_{t},\frac{\Delta }{d(\alpha(\xbf_{t})+\beta(\xbf_{t},3\delta)+1)\sqrt{T}})}}\right\}
\end{aligned}.
 $$
Then with a probability of at least $1-p$, for all $t\in[T]$, $f_\delta(\xbf_t)-f^{\star}\leq 2\Delta $ holds where $\Delta >f_\delta(\xbf_1)-f^{\star}$. Furthermore, we have 
\[
\sum_{t=1}^{T}\frac{\eta_{t}}{\sum_{s=1}^{T}\eta_{s}}\norm{\nabla f_{\delta}(\xbf_{t})}^{2}\leq\frac{2\Delta \log\left(\frac{2T}{p}\right)}{C_{\text{path}}\sqrt{T}},
\]
where $C_{\text{path}}=\frac{1}{T}\sum_{t=1}^{T}C_{t} .$
Consequently, the overall sample complexity for achieving $\min_{t\in[T]}\norm{\nabla f_{\delta}(\xbf_{t})}^{2}\leq\epsilon^{2}$ 
is $\tilde{\Ocal}(\epsilon^{-4}\delta^{-1}d^{\frac{5}{2}})$. 


\medskip

\begin{proof}
From the descent property in Lemma~\ref{lem:smoothness_descent}, for any $t\in[T]$, we have:
\begin{equation}
\begin{aligned}
f_{\delta}(\xbf_{t+1}) & \leq f_{\delta}(\xbf_{t})-\eta_{t}\norm{\nabla f_{\delta}(\xbf_{t})}^{2}+\eta_{t}\inner{\nabla f_{\delta}(\xbf_{t})-\gbf(\xbf_t,\wbf_t)}{\nabla f_{\delta}(\xbf_{t})}\\
&\quad +\frac{\eta_{t}^{2}\ell(\xbf_{t},\eta_{t}\norm{\gbf(\xbf_t,\wbf_t)})\norm{\gbf(\xbf_t,\wbf_t)}^{2}}{2}.
\end{aligned}
\label{eq:descent_one_step_re}
\end{equation}
Rearranging and summing from $s=1$ to $T$, we get:
\begin{equation}
\begin{aligned}
\sum_{s=1}^{T}\eta_{s}\norm{\nabla f_{\delta}(\xbf_{s})}^{2} & \leq f_{\delta}(\xbf_{1}) - f_{\delta}(\xbf_{T+1}) + \underbrace{\sum_{s=1}^{T}\eta_{s}\inner{\nabla f_{\delta}(\xbf_{s})-\gbf(\xbf_s,\wbf_s)}{\nabla f_{\delta}(\xbf_{s})}}_{\Acal_T} \\
& \quad + \underbrace{\sum_{s=1}^{T}\frac{\eta_{s}^{2}\ell(\xbf_{s},\eta_{s}\norm{\gbf(\xbf_s,\wbf_s)})\norm{\gbf(\xbf_s,\wbf_s)}^{2}}{2}}_{\Bcal_T}.
\end{aligned}
\label{eq:sum_main_re}
\end{equation}
Since $f_\delta(\xbf_{T+1}) \ge f^{\star}$ and by definition $\Delta > f_\delta(\xbf_1) - f^{\star}$, we can let $\Delta_{1} = f_\delta(\xbf_1) - f^{\star} < \Delta$, which gives:
\begin{equation}
\sum_{s=1}^{T}\eta_{s}\norm{\nabla f_{\delta}(\xbf_{s})}^{2} \leq \Delta_{1} + \Acal_T + \Bcal_T.
\label{eq:sum_main_bound_re}
\end{equation}

Our goal is to show that with high probability, the iterates $\{\xbf_t\}_{t=1}^T$ remain in a bounded region where the terms $\Acal_T$ and $\Bcal_T$ are well-controlled. 

\textbf{Bounds for Term $\Bcal_T$:}
For any $s \in [T]$, the stepsize $\eta_s$ is chosen to satisfy one of its bounds:
\[
\eta_{s}\leq\frac{1}{d(\alpha(\xbf_{s})+\beta(\xbf_{s},3\delta))}\sqrt{\frac{\Delta }{T\ell(\xbf_{s},\frac{1}{(\alpha(\xbf_{s})+\beta(\xbf_{s},3\delta)+1)\sqrt{T}})}},
\]
which then implies
\[
\frac{\eta_{s}^{2}\ell(\xbf_{s},\eta_{s}\norm{\gbf(\xbf_s,\wbf_s)})\norm{\gbf(\xbf_s,\wbf_s)}^{2}}{2} \leq \frac{\eta_{s}^{2}\ell\left(\xbf_{s},\dots\right) d^2(\alpha(\xbf_s)+\beta(\xbf_s,3\delta))^2}{2} \leq \frac{\Delta}{2T}.
\]
Summing from $s=1$ to $T$, we obtain:
\begin{equation}\label{eq:bounds_B}
    \Bcal_T = \sum_{s=1}^{T} \frac{\eta_{s}^{2}\ell(\xbf_{s},\eta_{s}\norm{\gbf(\xbf_s,\wbf_s)})\norm{\gbf(\xbf_s,\wbf_s)}^{2}}{2} \leq \sum_{s=1}^{T} \frac{\Delta}{2T} = \frac{\Delta}{2}.
\end{equation}

\textbf{Bounds for Term $\Acal_T$:} Here, we are going to show that term $A_t$ is bounded by $\frac{\Delta}{2}$ for any $t\leq T$ with probability at least $1-\frac{p}{T}$.
First, it is easy to see that the term $A_s$ sums up the martingale differences:  let $X_s = \eta_s \inner{\nabla f_{\delta}(\xbf_{s})-\gbf(\xbf_s,\wbf_s)}{\nabla f_{\delta}(\xbf_{s})}$ and take expectation on $\wbf_s$, we have $\Expe_{\wbf_s}[X_s] = 0$.

 The stepsize choice ensures:
\[
\eta_{s} \le \frac{C_s}{\log(2/p)\sqrt{T}} \le \frac{1}{\log(2/p)\sqrt{T}} \frac{\Delta}{2(d+1)(\alpha(\xbf_s)+\beta(\xbf_s,3\delta)+1)^2}.
\]
This implies an almost sure bound on each term $X_s$:
\[
|X_s| \le \eta_s \norm{\nabla f_\delta(\xbf_s) - \gbf(\xbf_s, \wbf_s)} \norm{\nabla f_\delta(\xbf_s)} \le \eta_s (d+1)(\alpha(\xbf_s)+\beta(\xbf_s,3\delta))^{2} \le \frac{\Delta}{2\log(2/p)}.
\]
Let $c = \frac{\Delta}{2\log(2T/p)}$. The sum of conditional variances is bounded by $\sum_{s=1}^{T}\Expe[X_s^2|\mathcal{F}_{s-1}] \le G = \frac{\Delta^2}{24\log(2T/p)}$.

We apply the Bernstein inequality for martingales (Lemma~\ref{lem:bernstein_lemma}) to the sum $\Acal_t = \sum_{s=1}^t X_s$. With $b = \Delta/2$, we get:
\begin{equation}
\Pbb\left( \left| \sum_{s=1}^{t} X_s \right| > \frac{\Delta}{2} \right) \le 2\exp\left(-\frac{(\Delta/2)^2}{2(G + c \cdot \Delta/2 / 3)}\right) \le 2\exp\left(-\frac{\Delta^2/4}{2G + \frac{\Delta^2}{12\log(2T/p)}}\right) \leq \frac{p}{T}.
\label{eq:bernstein_bound}
\end{equation}
This inequality shows that the event $|\Acal_T| \le \Delta/2$ holds with probability at least $1-p$.

Define the set $Q = \{ \xbf : f_{\delta}(\xbf) \leq f^{\star} + 2\Delta \}$. 
We aim to show that $\xbf_t \in Q$ for all $t \in [T]$ with probability at least $1 - p$.

To achieve this, define the event $E_t$ as follows:
\begin{align*}
\sum_{s=1}^{t} X_s
+ \sum_{s=1}^{t} \frac{\eta_s^2 \ell\left( \xbf_s, \eta_s \norm{\gbf(\xbf_s, \wbf_s)} \right) \norm{\gbf(\xbf_s, \wbf_s)}^2}{2} &\leq \Delta, \\
\Delta_{t} = f_{\delta}(\xbf_{t}) - f^{\star} & \leq 2\Delta.    
\end{align*}
When the event $E_t$ occurs, it follows that $\xbf_t \in Q$. 

We aim to show that $E_t$ holds with probability at least $1 - \frac{p}{T}$, for all $t \in [T]$. This implies that the intersection $\bigcap_{t=1}^{T} E_t$ holds with probability at least $1 - p$, which in turn ensures that $\xbf_t \in Q$ for all $t \in [T]$ with probability at least $1 - p$.

Clearly, $E_1$ holds because $\Delta_1 \leq \Delta < 2\Delta$. For any $t_0 \in 2,\dots,T$, summing \eqref{eq:descent_one_step_re} from $s=1$ to $t_0-1$ yields:
\[
f_{\delta}(\xbf_{t_0}) - f^{\star} \leq \Delta_1 + \sum_{s=1}^{t_0-1} X_s + \sum_{s=1}^{t_0-1} \frac{\eta_s^2 \ell\left( \xbf_s, \eta_s \norm{\gbf(\xbf_s, \wbf_s)} \right) \norm{\gbf(\xbf_s, \wbf_s)}^2}{2}.
\]

Since $\Delta_1 \leq \Delta$, Equation~\eqref{eq:bounds_B} implies that the second summation term is at most $\frac{\Delta}{2}$. Additionally, by Equation~\eqref{eq:bernstein_bound}, we have 
\[
\left| \sum_{s=1}^{t_0-1} X_s \right| \leq \frac{\Delta}{2}
\]
with probability at least $1 - \frac{p}{T}$. Therefore, the event $E_{t_0}$ holds with probability at least $1 - \frac{p}{T}$.

Applying a union bound, the event $\bigcap_{t=1}^T E_t$ holds with probability:
\begin{equation}
\mathbb{P}\left( \bigcap_{t=1}^T E_t \right) = 1- \mathbb{P}\left( \bigcup_{t=1}^T E_t^c \right) \ge 1- \sum_{t=1}^T \mathbb{P}(E_t^c) \geq 1 - \sum_{t=2}^T \frac{p}{T} \geq 1 - p,
\end{equation}
ensuring $\mathbf{x}_t \in Q$ for all $t \in [T]$.

We have shown that with probability of at least $1-p$, event $E_t$ occurs and all iterates $\xbf_t$ remain in $Q$, and our bounds on $A_t$ and $B_t$ hold for the full sum up to $T$.
Substituting these bounds into \eqref{eq:sum_main_bound_re}, we get:
\[
\sum_{s=1}^{T}\eta_{s}\norm{\nabla f_{\delta}(\xbf_{s})}^{2} \leq \Delta_1 + \Acal_T + \Bcal_T \leq \Delta + \frac{\Delta}{2} + \frac{\Delta}{2} = 2\Delta.
\]
Dividing by $\sum_{s=1}^{T}\eta_{s}$ yields:
\[
 \sum_{t=1}^{T}\frac{\eta_{t}}{\sum_{s=1}^{T}\eta_{s}}\norm{\nabla f_{\delta}(\xbf_{t})}^{2} \leq \frac{2\Delta }{\sum_{s=1}^{T}\eta_{s}} = \frac{2\Delta \log\left(\frac{2}{p}\right)}{\frac{1}{\sqrt{T}}\sum_{s=1}^T C_s} = \frac{2\Delta \log\left(\frac{2}{p}\right)}{C_{\text{path}}\sqrt{T}},
\]
where $C_{\text{path}} = \frac{1}{T}\sum_{t=1}^{T}C_{t}$ is the average path-dependent constant. 
We note that since all iterates remain in the bounded set $Q$, the functions $\alpha(\xbf_t)$, $\beta(\xbf_t, r)$, $\sigma(\xbf_t)$, and $\ell(\xbf_t, \cdot)$ are evaluated on a bounded set, ensuring they are well-behaved and making $C_{\text{path}}$ well-defined.
The analysis of sample complexity then follows.
\end{proof}

\subsection{Convergence rate of Algorithm \ref{alg:RS-NGF}} \label{sec:nsgd}
 This is the formal version of Theorem \ref{thm:result_nsgd}.

\textbf{Theorem 2}$\quad$  Under Assumption~\ref{assu:level_set}, let the batch
size be $B_{t}=\left\lceil \frac{T\sigma^{2}(\xbf_{t})}{\ell(\xbf_{t},\frac{\Delta }{(\alpha(\xbf_{t})+\beta(\xbf_{t},\delta)+1)\sqrt{T}})}\right\rceil$, 
stepsize
$\eta_{t}  \leq\frac{1}{\log(\frac{2T}{p})\sqrt{T}}C_{t},$
and 
$C_{t}=\min\left\{ \frac{\Delta }{12(\alpha(\xbf_{t})+\beta(\xbf_{t},\delta)+1)},\sqrt{\frac{\Delta }{3\ell(\xbf_{t},\frac{\Delta }{(\alpha(\xbf_{t})+\beta(\xbf_{t},\delta)+1)\sqrt{T}})}},\frac{2\Delta }{\sqrt{\ell(\xbf_{t},\frac{\Delta }{(\alpha(\xbf_{t})+\beta(\xbf_{t},\delta)+1)\sqrt{T}})}}\right\}$ where $p\in(0,1)$.
Then with a probability of at least $1-p$, for all $t\in[T]$, $f_\delta(\xbf_t)-f^{\star}\leq 2\Delta $ holds where $\Delta >f_\delta(\xbf_1)-f^{\star}$. Furthermore, we have 
\[
\sum_{t=1}^{T}\frac{\eta_{t}}{\sum_{s=1}^{T}\eta_{s}}\norm{\nabla f_{\delta}(\xbf_{t})}\leq\frac{2\Delta \log(\frac{2T}{p})}{\sqrt{T}C_{\text{path}}},
\]
 where $C_{\text{path }}=\frac{1}{T}\sum_{t=1}^{T}C_{t}.$ Consequently, to achieve $\min_{t\in[T]}\norm{\nabla f_{\delta}(\xbf_{t})}\leq\epsilon$, the total iteration number $T$  is $\tilde{\Ocal}(d^\frac{1}{2}\delta^{-1}\epsilon^{-2})$ and 
the overall sample complexity is $\tilde{\Ocal}(d^{\frac{3}{2}}\epsilon^{-4}\delta^{-1})$ 
. Then $\tilde{\xbf}\in\argmin_{\xbf_t,t\in[T]}\norm{\nabla f_{\delta}(\xbf_{t})}$ is a $(\delta,\epsilon)$-Goldstein stationary point of $f$. 

\medskip

\begin{proof}
Denote $\phi(\gbf(\xbf_{t},S_{B_{t}}))=\frac{\inner{\nabla f_{\delta}(\xbf_{t})}{\gbf(\xbf_{t},S_{B_{t}})}}{\norm{\gbf(\xbf_{t},S_{B_{t}})}\norm{\nabla f_{\delta}(\xbf_{t})}}$.
Invoking Lemma~\ref{lem:smoothness_descent}, we have

\begin{equation}
\begin{aligned}f_{\delta}(\xbf_{t+1}) & \aleq f_{\delta}(\xbf_{t})+\inner{\nabla f_{\delta}(\xbf_{t})}{\xbf_{t+1}-\xbf_{t}}+\frac{\ell(\xbf_{t},\norm{\xbf_{t+1}-\xbf_{t}})}{2}\norm{\xbf_{t+1}-\xbf_{t}}^{2}\\
 & =f_{\delta}(\xbf_{t})-\eta_{t}\left(\phi(\gbf(\xbf_{t},S_{B_{t}}))-\Expe_{S_{B_{s}}}[\phi(\gbf(\xbf_{t},S_{B_{t}}))]\right)\norm{\nabla f_{\delta}(\xbf_{t})}\\
 & \quad +\frac{\eta_{t}^{2}\ell(\xbf_{t},\eta_{t})}{2}-\eta_{t}\Expe_{S_{B_{s}}}[\phi(\gbf(\xbf_{t},S_{B_{t}}))]\norm{\nabla f_{\delta}(\xbf_{t})}\\
 & \bleq f_{\delta}(\xbf_{t})-\eta_{t}\left(\phi(\gbf(\xbf_{t},S_{B_{t}}))-\Expe_{S_{B_{s}}}[\phi(\gbf(\xbf_{t},S_{B_{t}}))]\right)\norm{\nabla f_{\delta}(\xbf_{t})}\\
 & \quad+\frac{\eta_{t}^{2}\ell(\xbf_{t},\eta_{t})}{2}-\eta_{t}\norm{\nabla f_{\delta}(\xbf_{t})}+2\eta_{t}\Expe_{S_{B_{s}}}\left[\norm{\nabla f_{\delta}(\xbf_{t})-\gbf(\xbf_{t},S_{B_{t}})}\right]\\
 & \cleq f_{\delta}(\xbf_{t})-\eta_{t}\norm{\nabla f_{\delta}(\xbf_{t})}-\eta_{t}\left(\phi(\gbf(\xbf_{t},S_{B_{t}}))-\Expe_{S_{B_{s}}}[\phi(\gbf(\xbf_{t},S_{B_{t}}))]\right)\norm{\nabla f_{\delta}(\xbf_{t})}\\
 & \quad+\frac{\eta_{t}^{2}\ell(\xbf_{t},\eta_{t})}{2}+\frac{2\eta_{t}\sigma(\norm{\xbf_{t}})}{\sqrt{B_{t}}}
\end{aligned}
\label{eq:descent_normalization-1-2-1}
\end{equation}
where $(a)$ comes from Lemma \ref{lem:smoothness_descent} , $(b)$ comes
from Lemma \ref{lem:huber} and $(c)$ comes from
\eqref{eq:bounds_B} and Jensen's inequality.

By telescoping, we have 
\begin{equation}
\begin{aligned}f_{\delta}(\xbf_{t+1}) & \leq f_{\delta}(\xbf_{1})-\eta_{t}\norm{\nabla f_{\delta}(\xbf_{t})}-\sum_{s=1}^{t}\eta_{s}\left(\phi(\gbf(\xbf_{s},S_{B_{s}}))-\Expe_{S_{B_{s}}}[\phi(\gbf(\xbf_{s},S_{B_{s}}))]\right)\norm{\nabla f_{\delta}(\xbf_{s})}\\
 & +\sum_{s=1}^{t}\frac{\eta_{s}^{2}\ell(\xbf_{s},\eta_{s})}{2}+2\sum_{s=1}^{t}\eta_{s}\sigma(\xbf_{s}).
\end{aligned}
\label{eq:descent_summation-2-1}
\end{equation}
Applying \eqref{eq:descent_summation-2-1} and rearranging the terms, we obtain
\begin{equation}
\begin{aligned}0\leq\sum_{s=1}^{t-1}\eta_{s}\norm{\nabla f_{\delta}(\xbf_{s})} & \aleq\Delta_{1}-\Delta_{t}-\sum_{s=1}^{t-1}\eta_{s}\left(\phi(\gbf(\xbf_{s},S_{B_{s}}))-\Expe_{S_{B_{s}}}[\phi(\gbf(\xbf_{s},S_{B_{s}}))]\right)\norm{\nabla f_{\delta}(\xbf_{s})}\\
 & \quad +\sum_{s=1}^{t-1}\frac{\eta_{s}^{2}\ell(\xbf_{s},\eta_{s})}{2}+\sum_{s=1}^{t-1}\frac{2\eta_{s}\sigma(\xbf_{s})}{B_{s}}
\end{aligned}
\label{eq:descent_summation-1-1-1-2-1}
\end{equation}
where $(a)$ comes from $f_\delta(\xbf_{t})-f_\delta(\xbf_{1})\leq\Delta_{1}-\Delta_{t}$. 

Rearranging the terms, we have
\begin{equation}\label{eq:total_term_nsgd}
    \begin{aligned}\Delta_{T} & \leq\Delta_{1}+\underbrace{\sum_{s=1}^{T-1}-\eta_{s}\left(\phi(\gbf(\xbf_{s},S_{B_{s}}))-\Expe_{S_{B_{s}}}[\phi(\gbf(\xbf_{s},S_{B_{s}}))]\right)\norm{\nabla f_{\delta}(\xbf_{s})}}_{\Acal_T}\\
 & \quad +\underbrace{\sum_{s=1}^{T-1}\frac{\eta_{s}^{2}\ell(\xbf_{s},\eta_{s})}{2}}_{\Bcal_T}+\underbrace{\sum_{s=1}^{T-1}\frac{2\eta_{s}\sigma(\xbf_{s})}{B_{s}}}_{\Ccal_T}
\end{aligned}
\end{equation}

In the following part, we are going to show that the term $\Acal_t$ is bounded
by $\frac{\Delta }{3}$ with high probability and both
$\Bcal_t$ and $\Ccal_t$ are bounded by $\frac{\Delta }{3}$ for any $t\leq T$.

\textbf{Bound for term $\Acal_t$ in \eqref{eq:total_term_nsgd}}
Here, we are going to show that term $\Acal_t$ is bounded by $\frac{\Delta}{3}$ for any $t\leq T$ with probability at least $1-\frac{p}{T}$.
First, we will show that the term $\Acal_t$ is a sum of martingale differences.  Let $X_s = -\eta_{t}\left(\phi(\gbf(\xbf_{t},S_{B_{t}}))-\Expe_{S_{B_{s}}}[\phi(\gbf(\xbf_{t},S_{B_{t}}))]\right)\norm{\nabla f_{\delta}(\xbf_{t})}$. Conditioned on $\mathcal{F}_{s-1}$, we have $\Expe_{\wbf_s}[X_s] = 0$.

Obviously, for any $t$, we have 
\[
\Expe_{S_{B_{s}}}\left[-\eta_{t}\left(\phi(\gbf(\xbf_{t},S_{B_{t}}))-\Expe_{S_{B_{s}}}[\phi(\gbf(\xbf_{t},S_{B_{t}}))]\right)\norm{\nabla f_{\delta}(\xbf_{t})}\right]=0
\]
 and 
\begin{equation}
X_s\aleq2\eta_{t}(\alpha(\xbf_{t})+\beta(\xbf_{t},\delta))\bleq\frac{\Delta }{4\log\left(\frac{2T}{p}\right)},\label{eq:abs_bound_nablaf_g-1}
\end{equation}
 where $(a)$ comes from that $\abs{\phi(\gbf_{t})}$ belongs to $[-1,1]$
and $(b)$ comes from that $\eta_{t}\leq\frac{\Delta }{8\log(\frac{2T}{p})(\alpha(\xbf_{t})+\beta(\xbf_{t},\delta))}$.

 For the same
reason, we have the bound of the variance
\begin{equation}
\Expe_{S_{B_{s}}}\left[X_s^{2}\right]\leq4\eta_{t}^{2}(\alpha(\xbf_{t})+\beta(\xbf_{t},\delta))^{2}.\label{eq:variance_bound-1}
\end{equation}

Taking \eqref{eq:abs_bound_nablaf_g-1} and \eqref{eq:variance_bound-1} into Lemma~\ref{lem:bernstein_lemma}, with $b=\frac{\Delta }{3}$, $c=\frac{\Delta }{4\log\left(\frac{2T}{p}\right)}$, $\sigma^2_s=4\eta_{t}^{2}(\alpha(\xbf_{t})+\beta(\xbf_{t},\delta))^{2}$ and $G=\frac{\Delta ^{2}}{36\log\left(\frac{2T}{p}\right)}$,
it holds that $\sum_{s=1}^{t-1}4\eta_{s}^{2}(\alpha(\xbf_{t})+\beta(\xbf_{t},\delta))^{2}\leq G$. Consequently, we have
\begin{equation}\label{eq:nsgd_prob}
    \begin{aligned} & \Pbb\left\{ \left|\sum_{s=1}^{t-1}X_s\right|>\frac{\Delta }{3}\right\} \leq2\exp\left(-\frac{\Delta ^{2}}{9(2G+\frac{\Delta ^{2}}{18\log\left(\frac{2T}{p}\right)})}\right) \leq\frac{p}{T}.
\end{aligned}
\end{equation}

\textbf{Bound for the term $\Bcal_t$ in \eqref{eq:total_term_nsgd}}

We have 
\[
\begin{aligned}\sum_{s=1}^{t-1}\frac{\eta_{s}^{2}\ell(\xbf_{s},\eta_{s})}{2} & \aleq\sum_{s=1}^{t-1}\frac{\eta_{s}^{2}\ell(\xbf_{s},\frac{\Delta }{(\alpha(\xbf_{s})+\beta(\xbf_{s},\delta)+1)\sqrt{T}})}{2} \bleq\frac{\Delta }{3},
\end{aligned}
\]
where $(a)$ comes from that $\eta_{s}\leq\frac{\Delta }{(\alpha(\xbf_{s})+\beta(\xbf_{s},\delta))\sqrt{T}}$
and $\beta(\xbf,r)$ is non-decreasing on $r$ and $(b)$ comes from that
$\eta_{s}\leq\sqrt{\frac{2\Delta }{3T\ell(\xbf_{s},\frac{\Delta }{(\alpha(\xbf_{s})+\beta(\xbf_{s},\delta)+1)\sqrt{T}})}}$.

\textbf{Bound for the term $\Ccal_t$ in \eqref{eq:total_term_nsgd} }

We have 
\[
\begin{aligned}\sum_{s=1}^{t-1}\frac{2\eta_{s}\sigma(\xbf_{s})}{\sqrt{B_{s}}} & \leq\frac{\Delta }{3},\end{aligned}
\]
where the inequality comes from that $\eta_{s}\leq\frac{2\Delta }{\sqrt{\ell(\xbf_{s},\frac{\Delta }{(\alpha(\xbf_{s})+\beta(\xbf_{s},\delta)+1)\sqrt{T}})T}}$
and the batch size $B_{s}=\left\lceil \frac{T\sigma^{2}(\xbf_{s})}{\ell(\xbf_{s},\frac{\Delta }{(\alpha(\xbf_{s})+\beta(\xbf_{s},\delta)+1)\sqrt{T}})}\right\rceil $.

Define the set 
\[
Q = \left\{ \xbf \in \Rbb^d : f_{\delta}(\xbf) \leq f^{\star} + 2\Delta \right\}.
\]
Our goal is to show that $\xbf_t \in Q$ for all $t \in [T]$ with probability at least $1 - p$.

To achieve this, define the event $E_t$ as follows:
\begin{align}
- \sum_{s=1}^{t-1} \eta_s \left( \phi(\gbf(\xbf_s, S_{B_s})) - \Expe_{S_{B_s}}[\phi(\gbf(\xbf_s, S_{B_s}))] \right) \norm{\nabla f_{\delta}(\xbf_s)} + \sum_{s=1}^{t-1} \frac{\eta_s^2 \ell(\xbf_s, \eta_s)}{2} + \sum_{s=1}^{t-1} \frac{2 \eta_s \sigma(\xbf_s)}{B_s} &\leq \Delta,  \\
\Delta_t = f_{\delta}(\xbf_t) - f^{\star} &\leq 2\Delta. 
\end{align}
We aim to show that $E_t$ holds for all $t \in [T]$ with probability at least $1 - \frac{p}{T}$. This implies that the intersection $\bigcap_{t=1}^T E_t$ holds with probability at least $1 - p$, ensuring that $\xbf_t \in Q$ for all $t \in [T]$ with probability at least $1 - p$.

For $t_0 \in [T]$, we have:
\begin{equation}
f_{\delta}(\xbf_{t_0}) - f^{\star} \leq \Delta_1 + \Acal_{t_0-1} + \Bcal_{t_0-1} + \Ccal_{t_0-1}, \label{eq:f_bound}
\end{equation}
where $\Delta_1 \leq \Delta$, and the terms satisfy $\Acal_{t_0-1} \leq \frac{\Delta}{3}$ with probability at least $1 - \frac{p}{T}$, $\Bcal_{t_0-1} \leq \frac{\Delta}{3}$, and $\Ccal_{t_0-1} \leq \frac{\Delta}{3}$ . Therefore,
\[
f_{\delta}(\xbf_{t_0}) - f^{\star} \leq \Delta + \frac{\Delta}{3} + \frac{\Delta}{3} + \frac{\Delta}{3} = 2\Delta,
\]
and the error sum satisfies $\Acal_{t_0-1} + \Bcal_{t_0-1} + \Ccal_{t_0-1} \leq \Delta$ with high probability. Thus, the event $E_{t_0}$ holds with probability at least $1 - \frac{p}{T}$.

Applying the union bound, we have:
\begin{equation}
\Pbb\left( \bigcap_{t=1}^T E_t \right) = 1 - \Pbb\left( \bigcup_{t=1}^T E_t^c \right) \geq 1 - \sum_{t=1}^T \Pbb(E_t^c) \geq 1 - \sum_{t=1}^T \frac{p}{T} \geq 1 - p, \label{eq:union_bound2}
\end{equation}
ensuring that $\xbf_t \in Q$ for all $t \in [T]$ with probability at least $1 - p$.

Thus, with probability at least $1 - p$, the event $E_t$ occurs, all iterates $\xbf_t$ remain in $Q$, and the bounds on terms $A_t$, $B_t$, and $C_t$ hold for the full sum up to $T$. We have:
\begin{equation}
\sum_{s=1}^T \eta_s \norm{\nabla f_{\delta}(\xbf_s)} \leq 2\Delta. \label{eq:nsgd_final}
\end{equation}
Dividing both sides of Equation~\eqref{eq:nsgd_final} by $\sum_{s=1}^T \eta_s$, we obtain:
\begin{align}
\sum_{t=1}^T \frac{\eta_t}{\sum_{s=1}^T \eta_s} \norm{\nabla f_{\delta}(\xbf_t)} &\leq \frac{2\Delta}{\sum_{s=1}^T \eta_s} \leq \frac{2\Delta \log\left( \frac{2T}{p} \right)}{C_{\text{path}} \sqrt{T}}, \label{eq:weighted_gradient}
\end{align}
where $C_{\text{path}} = \frac{1}{T} \sum_{t=1}^T C_t$ and $C_t = \Ocal(\Delta d^{-\frac{1}{4}} \delta^{\frac{1}{2}})$. The batch sizes are set as:
\[
\left\lceil \frac{T \sigma^2(\xbf_t)}{\ell\left( \xbf_t, \frac{\Delta}{\left( \alpha(\xbf_t) + \beta(\xbf_t, \delta) + 1 \right) \sqrt{T}} \right)} \right\rceil = \tilde{\Ocal}(d^{\frac{1}{2}} \epsilon^{-2}),
\]
where $\sigma^2(\xbf_t) = \Ocal(d)$ and $\ell\left( \xbf_t, \frac{\Delta}{\left( \alpha(\xbf_t) + \beta(\xbf_t, \delta) + 1 \right) \sqrt{T}} \right) = \Ocal(\delta^{-1} d^{\frac{1}{2}})$.

To achieve $\min_{t \in [T]} \norm{\nabla f_{\delta}(\xbf_t)} \leq \epsilon$, the number of iterations required is $\Ocal(d^{\frac{1}{2}} \delta^{-1} \epsilon^{-2})$, and the overall sample complexity is $\tilde{\Ocal}(\epsilon^{-4} \delta^{-1} d^{\frac{3}{2}})$. 

\end{proof}
\subsection{Convergence rate of Algorithm~\ref{alg:storm}}\label{sec:nsvrg1}

This is the formal version of Theorem~\ref{thm:result_nsvrg}.

\medskip

\textbf{Theorem~\ref{thm:result_nsvrg}}  Under Assumption~\ref{assu:level_set}, let $B_{s_{0}}=\left\lceil \frac{72\sigma^{2}(\xbf_{s_{0}})T\delta}{\sqrt{d}}\right\rceil $ where $t_{0}\bmod q=0$, $b_{s}=\left\lceil 72qd\right\rceil$,
$q=\left\lceil \epsilon^{-1}\right\rceil $ ,
the stepsize $\eta_{t}\leq\frac{1}{\log\left(\frac{2T}{p}\right)\sqrt{T}}C_{t},$
where $p\in(0,1)$ and 
\begin{equation*}
    C_{t}=\min\left\{\frac{\delta^\frac{1}{2}}{d^\frac{1}{4}} ,\frac{\Delta }{24(\alpha(\xbf_{t})+\beta(\xbf_{t},\delta)+1)},\frac{\Delta \delta^{\frac{1}{2}}}{d^\frac{1}{4}(\alpha(\xbf_{t})+\beta(\xbf_{t},\delta +\frac{\Delta }{\sqrt{T}(\alpha(\xbf_{t})+\beta(\xbf_{t},\delta)+1)}))},\sqrt{\frac{2\Delta }{3\ell(\xbf_{t},\frac{\Delta }{(\alpha(\xbf_{t})+\beta(\xbf_{t},\delta)+1)\sqrt{T}})}}\right\}.
\end{equation*}  
Then with a probability of at least $1-p$, for all $t\in[T]$, $f_\delta(\xbf_t)-f^{\star}\leq 2\Delta $ holds where $\Delta >f_\delta(\xbf_1)-f^{\star}$. Furthermore, the iterates generated by \eqref{alg:storm}
satisfy that
\[
\sum_{t=1}^{T}\frac{\eta_{t}}{\sum_{s=1}^{T}\eta_{s}}\norm{\nabla f_{\delta}(\xbf_{t})}\leq\frac{2\Delta \log\left(\frac{2T}{p}\right)}{C_{\text{path }}\sqrt{T}},
\]
where $C_{\text{path}}=\frac{1}{T}\sum_{t=1}^{T}C_{t}$. Consequently,
the overall sample complexity to achieve $\min_{t\in[T]}\norm{\nabla f_{\delta}(\xbf_{t})}\leq\epsilon$
is $\tilde{\Ocal}(d^{\frac{3}{2}}\epsilon^{-3}\delta^{-1})$. Then $\tilde{\xbf}\in\argmin_{\xbf_t,t\in[T]}\norm{\nabla f_{\delta}(\xbf_{t})}$ is a $(\delta,\epsilon)$-Goldstein stationary point of $f$.

\medskip

\begin{proof}
Denote $\phi(\mbf_{t})=\frac{\inner{\nabla f_{\delta}(\xbf_{t})}{\mbf_{t}}}{\norm{\mbf_{t}}\norm{\nabla f_{\delta}(\xbf_{t})}}$.
With the descent lemma \ref{lem:smoothness_descent}, we have
\begin{equation}
\begin{aligned}f_{\delta}(\xbf_{t+1}) & \aleq f_{\delta}(\xbf_{t})+\inner{\nabla f_{\delta}(\xbf_{t})}{\xbf_{t+1}-\xbf_{t}}+\frac{\ell(\xbf_{t},\norm{\xbf_{t+1}-\xbf_{t}})}{2}\norm{\xbf_{t+1}-\xbf_{t}}^{2}\\
 & =f_{\delta}(\xbf_{t})-\eta_{t}\left(\phi(\mbf_{t})-\Expe_{\mbf_{t}}[\phi(\mbf_{t})]\right)\norm{\nabla f_{\delta}(\xbf_{t})}\\
 & \quad +\frac{\eta_{t}^{2}\ell(\xbf_{t},\eta_{t})}{2}-\eta_{t}\Expe_{\mbf_{t}}[\phi(\mbf_{t})]\norm{\nabla f_{\delta}(\xbf_{t})}\\
 & \bleq f_{\delta}(\xbf_{t})-\eta_{t}\left(\phi(\mbf_{t})-\Expe_{\mbf_{t}}[\phi(\mbf_{t})]\right)\norm{\nabla f_{\delta}(\xbf_{t})}\\
 & \quad +\frac{\eta_{t}^{2}\ell(\xbf_{t},\eta_{t})}{2}-\eta_{t}\norm{\nabla f_{\delta}(\xbf_{t})}+2\eta_{t}\Expe_{\mbf_{t}}\left[\norm{\nabla f_{\delta}(\xbf_{t})-\mbf_{t}}\right]\\
 & \cleq f_{\delta}(\xbf_{t})-\eta_{t}\left(\phi(\mbf_{t})-\Expe_{\mbf_{t}}[\phi(\mbf_{t})]\right)\norm{\nabla f_{\delta}(\xbf_{t})}\\
 & \quad +\frac{\eta_{t}^{2}\ell(\xbf_{t},\eta_{t})}{2}-\eta_{t}\norm{\nabla f_{\delta}(\xbf_{t})}\\
 & \quad +2\eta_{t}\sqrt{\Expe_{\mbf_{t}}\left[\norm{\mbf_{t}-\nabla f_{\delta}(\xbf_{t})}^{2}\right]},
\end{aligned}
\label{eq:descent_normalization-1-2-1-1}
\end{equation}
where $(a)$ comes from Lemma \ref{lem:smoothness_descent}, $(b)$ comes
from Lemma \ref{lem:huber} and $(c)$ comes from Jensen's inequality.

Summing up \eqref{eq:descent_normalization-1-2-1-1} from 1 to $t$, we have 
\begin{equation}
\begin{aligned}f_{\delta}(\xbf_{t+1}) & \leq f_{\delta}(\xbf_{1})-\sum_{s=1}^{t}\eta_{s}\norm{\nabla f_{\delta}(\xbf_{s})}+\sum_{s=1}^{t}\eta_{s}\left(\phi(\mbf_{s})-\Expe_{\mbf_{t}}[\phi(\mbf_{s})]\right)\norm{\nabla f_{\delta}(\xbf_{s})}\\
 & \quad +\sum_{s=1}^{t}\frac{\eta_{s}^{2}\ell(\xbf_{s},\eta_{s})}{2}+\sum_{s=1}^{t}2\eta_{s}\sqrt{\Expe_{\mbf_{t}}\left[\norm{\mbf_{s}-\nabla f_{\delta}(\xbf_{s})}^{2}\right]}.
\end{aligned}
\label{eq:descent_summation-2-1-1}
\end{equation}

Then we apply \eqref{eq:descent_summation-2-1-1} and rearrange the
terms to obtain
\begin{equation}
\begin{aligned}0\leq\sum_{s=1}^{t-1}\eta_{s}\norm{\nabla f_{\delta}(\xbf_{s})} & \aleq\Delta_{1}-\Delta_{t}-\sum_{s=1}^{t-1}\eta_{s}\left(\phi(\mbf_{s})-\Expe[\phi(\mbf_{s})]\right)\norm{\nabla f_{\delta}(\xbf_{s})}+\sum_{s=1}^{t-1}\frac{\eta_{s}^{2}\ell(\xbf_{s},\eta_{s})}{2}\\
 & \quad  +2\sum_{s=1}^{t-1}\eta_{s}\sqrt{\Expe\left[\norm{\mbf_{s}-\nabla f_{\delta}(\xbf_{s})}^{2}\right]},
\end{aligned}
\label{eq:descent_summation-1-1-1-2-1-1}
\end{equation}
where $(a)$ comes from $f_\delta(\xbf_{t})-f_\delta(\xbf_{1})\leq\Delta_{1}-\Delta_{t}$. 

Rearranging the terms, we have
\begin{equation}\label{eq:total_term_svrg}
    \begin{aligned}\Delta_{t} & \leq\Delta_{1}\underbrace{-\sum_{s=1}^{t-1}\eta_{s}\left(\phi(\mbf_{s})-\Expe[\phi(\mbf_{s})]\right)\norm{\nabla f_{\delta}(\xbf_{s})}}_{\Acal_t}\\
 & \quad  +\underbrace{\sum_{s=1}^{t-1}2\eta_{s}\sqrt{\Expe\left[\norm{\mbf_{s}-\nabla f_{\delta}(\xbf_{s})}^{2}\right]}}_{\Bcal_t}+\underbrace{\sum_{s=1}^{t-1}\frac{\eta_{s}^{2}\ell(\norm{\xbf_{s}},\norm{\xbf_{s}}+\eta_{s})}{2}}_{\Ccal_t}
\end{aligned}
\end{equation}

In the following part, we are going to show that the term $A_t$ is bounded
by $\frac{\Delta }{3}$ with high probability and the term
$B_t$ and $C_t$ are bounded by $\frac{\Delta }{3}$  for any $t\leq T$ respectively.

\textbf{Bound for term $\Acal_t$ in \eqref{eq:total_term_svrg}:}

Obviously, for any $t$, we have 
\[
\Expe_{\mbf_{t}}\left[-\eta_{t}\left(\phi(\mbf_{t})-\Expe_{\mbf_{t}}[\phi(\mbf_{t})]\right)\norm{\nabla f_{\delta}(\xbf_{t})}\right]=0
\]
 and 
\begin{equation}
\begin{aligned}-\eta_{t}\left(\phi(\mbf_{t})-\Expe[\phi(\mbf_{t})]\right)\norm{\nabla f_{\delta}(\xbf_{t})} & \aleq2\eta_{t}(\alpha(\xbf_{s})+\beta(\xbf_{s},\delta))\bleq\frac{\Delta }{4\log\left(\frac{2T}{p}\right)},\end{aligned}
\label{eq:abs_bound_nablaf_g-1-1}
\end{equation}
 where $(a)$ comes from $\abs{\phi(\mbf_{t})}$ belongs to $[-1,1]$
and $(b)$ comes from that $\eta_{s}\leq\frac{\Delta}{8(\alpha(\xbf_{s})+\beta(\xbf_{s},\delta))\log\left(\frac{2T}{p}\right)}$

Furthermore, we have the variance bound

\begin{equation}
\Expe_{\mbf_{t}}\left[\eta_{t}^{2}\left(\phi(\mbf_{t})-\Expe_{\mbf_{t}}[\phi(\mbf_{t})]\right)^{2}\norm{\nabla f_{\delta}(\xbf_{t})}^{2}\right]\leq4\eta_{t}^{2}\norm{\nabla f_{\delta}(\xbf_{t})}^{2}.\label{eq:variance_bound-1-1}
\end{equation}

Taking \eqref{eq:abs_bound_nablaf_g-1-1} and \eqref{eq:variance_bound-1-1} into \ref{lem:bernstein_lemma}, with $b=\frac{\Delta }{3}$,
$c=\frac{\Delta }{4\log\left(\frac{2T}{p}\right)}$,
$\sigma_s^2=4\eta_{s}^{2}(\alpha(\xbf_{s})+\beta(\xbf_{s},\delta))^{2}$
and
$G=\frac{\Delta ^{2}}{36\log\left(\frac{2T}{p}\right)}$,
we have 
\begin{equation}\label{eq:nsvrg_prob}
    \begin{aligned} & \Pbb\left\{ \left|-\sum_{s=1}^{t-1}\eta_{s}\left(\phi(\mbf_{s})-\Expe_{\mbf_{t}}[\phi(\mbf_{s})]\right)\norm{\nabla f_{\delta}(\xbf_{s})}\right|>\frac{\Delta }{3},\sum_{s=1}^{t-1}4\eta_{s}^{2}(\alpha(\xbf_{s})+\beta(\xbf_{s},\delta))^{2}\leq G\right\} \\
 & \aleq2\exp\left(-\frac{\Delta ^{2}}{9(2G+\frac{\Delta ^{2}}{18\log\left(\frac{2T}{p}\right)})}\right)\\
 & \bleq\frac{p}{T}.
\end{aligned}
\end{equation}

Furthermore, we have 
\begin{equation}\label{eq:e1bar_nsvrg}
    \begin{aligned}\sum_{s=1}^{t-1}4\eta_{s}^{2}(\alpha(\xbf_{s})+\beta(\xbf_{s},\delta))^{2} & \aleq\frac{(t-1)\Delta ^{2}}{36T\log\left(\frac{2T}{p}\right)}\\
 & \bleq G.
\end{aligned}
\end{equation}
where $(a)$ comes from the fact that $\eta_{s}\leq\frac{\Delta }{12(\alpha(\xbf_{s})+\beta(\xbf_{s},\delta))}\sqrt{\frac{1}{T\log\left(\frac{2T}{p}\right)}}$
and $(b)$ comes from the definition of $G$.

\textbf{Bound for term $\Bcal_t$ in \eqref{eq:total_term_svrg}:}
\[
\begin{aligned} & \sum_{s=1}^{t-1}2\eta_{s}\Expe_{\mbf_{t}}\left[\norm{\mbf_{s}-\nabla f_{\delta}(\xbf_{s})}\right]\\
 & \aleq\sum_{s=1}^{t-1}2\eta_{s}\sqrt{\frac{\sigma^{2}(\xbf_{s_{0}})}{B_{s_{0}}}+\sum_{\tau=s_{0}+1}^{s}\frac{\eta_{\tau}^{2}}{b_{\tau}}\frac{d^{2}(\alpha(\xbf_{\tau})+\beta(\xbf_{\tau},\eta_{\tau}+\delta))^{2}}{\delta^{2}}}\\
 & \bleq\sum_{s=1}^{t-1}2\eta_{s}\sqrt{\frac{d^{\frac{1}{2}}}{72T\delta}+\frac{d^{\frac{1}{2}}}{72T\delta}}\\
 & \cleq\frac{\Delta}{3},
\end{aligned}
\]
where $(a)$ the Jensen's inequality and Lemma \ref{lem:sum_theta_b},
and $(b)$ comes from the batch size $B_{s_{0}}=\left\lceil \frac{72\sigma^{2}(\xbf_{s_{0}})T\delta}{\sqrt{d}}\right\rceil $,$b_{s}=\left\lceil 72qd\right\rceil $
and $\eta_{s}\leq\frac{\delta^{\frac{1}{2}}}{d^{\frac{1}{4}}(\alpha(\xbf_{s})+\beta(\xbf_{s},\frac{1}{(\alpha(\xbf_{s})+\beta(\xbf_{s},\delta)+1)\sqrt{T}}))\sqrt{T}}$
$(c)$ comes from that $\eta_{s}\leq\frac{\delta^{\frac{1}{2}}}{d^{\frac{1}{4}}\sqrt{T}}$.

\textbf{Bound for term $\Ccal_t$ in \eqref{eq:total_term_svrg}:}
\[
\begin{aligned}\sum_{s=1}^{t-1}\frac{\eta_{s}^{2}\ell(\xbf_{s},\eta_{s})}{2} & \aleq\sum_{s=1}^{t-1}\frac{\eta_{s}^{2}\ell(\xbf_{s},\frac{\Delta }{(\alpha(\xbf_{s})+\beta(\xbf_{s},\delta)+1)\sqrt{T}})}{2}\\
 & \bleq\frac{\Delta }{3},
\end{aligned}
\]
where $(a)$ comes from that $\eta_{s}\leq\frac{\Delta }{(\alpha(\xbf_{s})+\beta(\xbf_{s},\delta)+1)\sqrt{T}}$
and $\beta(\xbf,r)$ is non-decreasing on $r$, and $(b)$ comes from
that $\eta_{s}\leq\sqrt{\frac{2\Delta }{3T\ell(\xbf_{s},\frac{\Delta }{(\alpha(\xbf_{s})+\beta(\xbf_{s},\delta)+1)\sqrt{T}})}}$.

Define the event $E_t$ as:
\begin{align}
- \sum_{s=1}^{t-1} \eta_s \left( \phi(\gbf_s) - \Expe_{\gbf_s}[\phi(\gbf_s)] \right) \norm{\nabla f_{\delta}(\xbf_s)} + \sum_{s=1}^{t-1} \frac{\eta_s^2 \ell(\xbf_s, \eta_s)}{2} + 2 \sum_{s=1}^{t-1} \eta_s \sqrt{\Expe_{\gbf_s}\left[ \norm{\gbf_s - \nabla f_{\delta}(\xbf_s)}^2 \right]} &\leq \Delta, \\
\Delta_t = f_{\delta}(\xbf_t) - f^{\star} &\leq 2\Delta. 
\end{align}
When $E_t$ holds, $\xbf_t \in Q$, since $\Delta_t \leq 2\Delta$ implies $f_{\delta}(\xbf_t) \leq f^{\star} + 2\Delta$.

Since $\Acal_t \leq \frac{\Delta}{3}$, $\Bcal_t \leq \frac{\Delta}{3}$, and $\Ccal_t \leq \frac{\Delta}{3}$ hold with probability at least $1 - \frac{p}{T}$, we have $\Delta_t \leq \Delta_1 + \frac{\Delta}{3} + \frac{\Delta}{3} + \frac{\Delta}{3} \leq 2\Delta$ . Thus, $E_t$ holds with probability at least $1 - \frac{p}{T}$. By the union bound:
\begin{align}
\Pbb\left( \bigcap_{t=1}^T E_t \right) = 1 - \Pbb\left( \bigcup_{t=1}^T E_t^c \right) \geq 1 - \sum_{t=1}^T \Pbb(E_t^c) \geq 1 - \sum_{t=1}^T \frac{p}{T} \geq 1 - p, \label{eq:union_bound3}
\end{align}
ensuring that $\xbf_t \in Q$ for all $t \in [T]$ with probability at least $1 - p$.

We have 
\begin{equation}\label{eq:nsvrg_final}
    \sum_{s=1}^{T}\eta_{s}\norm{\nabla f_{\delta}(\xbf_{s})}^{2}\leq2\Delta .
\end{equation} 
Dividing both sides of \eqref{eq:nsvrg_final} by $\sum_{s=1}^{T}\eta_{s}$, we have
\[
\sum_{t=1}^{T}\frac{\eta_{t}}{\sum_{s=1}^{T}\eta_{s}}\norm{\nabla f_{\delta}(\xbf_{t})}\leq\frac{2\Delta \log\left(\frac{2T}{p}\right)}{C_{\text{path }}\sqrt{T}},
\]
Note that $C_{\text{path}}=\frac{1}{T}\sum_{t=1}^{T}C_{t} $ and  $C_{t}$ is $\Ocal(\Delta d^{-\frac{1}{4}}\delta^{\frac{1}{2}})$. The batch sizes  are set as $B_{s_{0}}=\left\lceil \frac{72\sigma^{2}(\xbf_{s_{0}})T\delta}{\sqrt{d}}\right\rceil=\tilde{\Ocal}(d\epsilon^{-2}) $ where $t_{0}\bmod q=0$, $\sigma^2(\xbf_t)=\Ocal(d)$. $b_{s}=\left\lceil 72qd\right\rceil=\Ocal(\epsilon^{-1}d)$, the period 
$q=\left\lceil \epsilon^{-1}\right\rceil $ . Then we can show that to achieve $\min_{t\in[T]}\norm{\nabla f_{\delta}(\xbf_{t})}\leq\epsilon$, the iteration times should be $\Ocal(d^\frac{1}{2}\delta^{-1}\epsilon^{-2})$ and the overall sample complexity 
is $\tilde{\Ocal}(\epsilon^{-3}\delta^{-1}d^{\frac{3}{2}})$.
\end{proof}

\end{document}

%% file: macros.tex
\global\long\def\inprod#1#2{\left\langle #1,#2\right\rangle }%
\global\long\def\inner#1#2{\langle#1,#2\rangle}%
\global\long\def\binner#1#2{\big\langle#1,#2\big\rangle}%
\global\long\def\Binner#1#2{\Big\langle#1,#2\Big\rangle}%
\global\long\def\norm#1{\Vert#1\Vert}%
\global\long\def\bnorm#1{\big\Vert#1\big\Vert}%
\global\long\def\Bnorm#1{\Big\Vert#1\Big\Vert}%
\global\long\def\abs#1{|#1|}%
\global\long\def\setnorm#1{\Vert#1\Vert_{-}}%
\global\long\def\bsetnorm#1{\big\Vert#1\big\Vert_{-}}%
\global\long\def\Bsetnorm#1{\Big\Vert#1\Big\Vert_{-}}%
\global\long\def\Vol{\operatorname{Vol}}%

\global\long\def\brbra#1{\big(#1\big)}%
\global\long\def\Brbra#1{\Big(#1\Big)}%
\global\long\def\rbra#1{(#1)}%

\global\long\def\sbra#1{[#1]}%
\global\long\def\bsbra#1{\big[#1\big]}%
\global\long\def\Bsbra#1{\Big[#1\Big]}%

\global\long\def\cbra#1{\{#1\}}%
\global\long\def\bcbra#1{\big\{#1\big\}}%
\global\long\def\Bcbra#1{\Big\{#1\Big\}}%
\global\long\def\vertiii#1{\left\vert \kern-0.25ex  \left\vert \kern-0.25ex  \left\vert #1\right\vert \kern-0.25ex  \right\vert \kern-0.25ex  \right\vert }%
\global\long\def\matr#1{\bm{#1}}%
\global\long\def\til#1{\tilde{#1}}%
\global\long\def\wtil#1{\widetilde{#1}}%
\global\long\def\wh#1{\widehat{#1}}%
\global\long\def\mcal#1{\mathcal{#1}}%
\global\long\def\mbb#1{\mathbb{#1}}%
\global\long\def\mtt#1{\mathtt{#1}}%
\global\long\def\ttt#1{\texttt{#1}}%
\global\long\def\dtxt{\textrm{d}}%
\global\long\def\bignorm#1{\bigl\Vert#1\bigr\Vert}%
\global\long\def\Bignorm#1{\Bigl\Vert#1\Bigr\Vert}%
\global\long\def\rmn#1#2{\mathbb{R}^{#1\times#2}}%
\global\long\def\deri#1#2{\frac{d#1}{d#2}}%
\global\long\def\pderi#1#2{\frac{\partial#1}{\partial#2}}%
\global\long\def\limk{\lim_{k\rightarrow\infty}}%
\global\long\def\trans{\textrm{T}}%
\global\long\def\onebf{\mathbf{1}}%
\global\long\def\zerobf{\mathbf{0}}%
\global\long\def\zero{\bm{0}}%

\global\long\def\Euc{\mathrm{E}}%
\global\long\def\Expe{\mathbb{E}}%
\global\long\def\rank{\mathrm{rank}}%
\global\long\def\range{\mathrm{range}}%
\global\long\def\diam{\mathrm{diam}}%
\global\long\def\epi{\mathrm{epi} }%
\global\long\def\relint{\mathrm{relint} }%
\global\long\def\inte{\operatornamewithlimits{int}}%
\global\long\def\cov{\mathrm{Cov}}%
\global\long\def\argmin{\operatornamewithlimits{arg\,min}}%
\global\long\def\argmax{\operatornamewithlimits{arg\,max}}%
\global\long\def\tr{\operatornamewithlimits{tr}}%
\global\long\def\dis{\operatornamewithlimits{dist}}%
\global\long\def\sign{\operatornamewithlimits{sign}}%

\global\long\def\prob{\mathrm{Prob}}%
\global\long\def\st{\operatornamewithlimits{s.t.}}%
\global\long\def\dom{\mathrm{dom}}%
\global\long\def\prox{\mathrm{prox}}%
\global\long\def\for{\mathrm{for}}%
\global\long\def\diag{\mathrm{diag}}%
\global\long\def\and{\mathrm{and}}%
\global\long\def\st{\mathrm{s.t.}}%
\global\long\def\dist{\mathrm{dist}}%
\global\long\def\Var{\operatornamewithlimits{Var}}%
\global\long\def\raw{\rightarrow}%
\global\long\def\law{\leftarrow}%
\global\long\def\Raw{\Rightarrow}%
\global\long\def\Law{\Leftarrow}%
\global\long\def\vep{\varepsilon}%
\global\long\def\dom{\operatornamewithlimits{dom}}%
\global\long\def\tsum{{\textstyle {\sum}}}%
\global\long\def\Cbb{\mathbb{C}}%
\global\long\def\Ebb{\mathbb{E}}%
\global\long\def\Fbb{\mathbb{F}}%
\global\long\def\Nbb{\mathbb{N}}%
\global\long\def\Rbb{\mathbb{R}}%
\global\long\def\Sbb{\mathbb{S}}%
\global\long\def\Xbb{\mathbb{X}}%
\global\long\def\extR{\widebar{\mathbb{R}}}%
\global\long\def\Pbb{\mathbb{P}}%
\global\long\def\Bbb{\mathbb{\mathbb{B}}}%
\global\long\def\Mrm{\mathrm{M}}%
\global\long\def\Acal{\mathcal{A}}%
\global\long\def\Bcal{\mathcal{B}}%
\global\long\def\Ccal{\mathcal{C}}%
\global\long\def\Dcal{\mathcal{D}}%
\global\long\def\Ecal{\mathcal{E}}%
\global\long\def\Fcal{\mathcal{F}}%
\global\long\def\Gcal{\mathcal{G}}%
\global\long\def\Hcal{\mathcal{H}}%
\global\long\def\Ical{\mathcal{I}}%
\global\long\def\Kcal{\mathcal{K}}%
\global\long\def\Lcal{\mathcal{L}}%
\global\long\def\Mcal{\mathcal{M}}%
\global\long\def\Ncal{\mathcal{N}}%
\global\long\def\Ocal{\mathcal{O}}%
\global\long\def\Pcal{\mathcal{P}}%
\global\long\def\Scal{\mathcal{S}}%
\global\long\def\Tcal{\mathcal{T}}%
\global\long\def\Xcal{\mathcal{X}}%
\global\long\def\Ycal{\mathcal{Y}}%
\global\long\def\Zcal{\mathcal{Z}}%
\global\long\def\i{i}%

\global\long\def\abf{\mathbf{a}}%
\global\long\def\bbf{\mathbf{b}}%
\global\long\def\cbf{\mathbf{c}}%
\global\long\def\fbf{\mathbf{f}}%
\global\long\def\gbf{\mathbf{g}}%
\global\long\def\lambf{\bm{\lambda}}%
\global\long\def\alphabf{\bm{\alpha}}%
\global\long\def\sigmabf{\bm{\sigma}}%
\global\long\def\thetabf{\bm{\theta}}%
\global\long\def\deltabf{\bm{\delta}}%
\global\long\def\sbf{\mathbf{s}}%
\global\long\def\lbf{\mathbf{l}}%
\global\long\def\ubf{\mathbf{u}}%
\global\long\def\vbf{\mathbf{v}}%
\global\long\def\mbf{\mathbf{\mathbf{m}}}%
\global\long\def\wbf{\mathbf{w}}%
\global\long\def\hbf{\mathbf{\mathbf{h}}}%
\global\long\def\mbf{\mathbf{m}}%
\global\long\def\xbf{\mathbf{x}}%
\global\long\def\ybf{\mathbf{y}}%
\global\long\def\zbf{\mathbf{z}}%
\global\long\def\Abf{\mathbf{A}}%
\global\long\def\Ubf{\mathbf{U}}%
\global\long\def\Pbf{\mathbf{P}}%
\global\long\def\Ibf{\mathbf{I}}%
\global\long\def\Ebf{\mathbf{E}}%
\global\long\def\Mbf{\mathbf{M}}%
\global\long\def\Qbf{\mathbf{Q}}%
\global\long\def\Lbf{\mathbf{L}}%
\global\long\def\Pbf{\mathbf{P}}%

\global\long\def\abm{\bm{a}}%
\global\long\def\bbm{\bm{b}}%
\global\long\def\cbm{\bm{c}}%
\global\long\def\dbm{\bm{d}}%
\global\long\def\ebm{\bm{e}}%
\global\long\def\fbm{\bm{f}}%
\global\long\def\gbm{\bm{g}}%
\global\long\def\hbm{\bm{h}}%
\global\long\def\pbm{\bm{p}}%
\global\long\def\qbm{\bm{q}}%
\global\long\def\rbm{\bm{r}}%
\global\long\def\sbm{\bm{s}}%
\global\long\def\tbm{\bm{t}}%
\global\long\def\ubm{\bm{u}}%
\global\long\def\vbm{\bm{v}}%
\global\long\def\wbm{\bm{w}}%
\global\long\def\xbm{\bm{x}}%
\global\long\def\ybm{\bm{y}}%
\global\long\def\zbm{\bm{z}}%
\global\long\def\Abm{\bm{A}}%
\global\long\def\Bbm{\bm{B}}%
\global\long\def\Cbm{\bm{C}}%
\global\long\def\Dbm{\bm{D}}%
\global\long\def\Ebm{\bm{E}}%
\global\long\def\Fbm{\bm{F}}%
\global\long\def\Gbm{\bm{G}}%
\global\long\def\Hbm{\bm{H}}%
\global\long\def\Ibm{\bm{I}}%
\global\long\def\Jbm{\bm{J}}%
\global\long\def\Lbm{\bm{L}}%
\global\long\def\Obm{\bm{O}}%
\global\long\def\Pbm{\bm{P}}%
\global\long\def\Qbm{\bm{Q}}%
\global\long\def\Rbm{\bm{R}}%
\global\long\def\Ubm{\bm{U}}%
\global\long\def\Vbm{\bm{V}}%
\global\long\def\Wbm{\bm{W}}%
\global\long\def\Xbm{\bm{X}}%
\global\long\def\Ybm{\bm{Y}}%
\global\long\def\Zbm{\bm{Z}}%
\global\long\def\lambm{\bm{\lambda}}%
\global\long\def\alphabm{\bm{\alpha}}%
\global\long\def\albm{\bm{\alpha}}%
\global\long\def\taubm{\bm{\tau}}%
\global\long\def\mubm{\bm{\mu}}%
\global\long\def\yrm{\mathrm{y}}%
\global\long\def\frechet{\text{{Fréchet}}}%
\global\long\def\lips{\text{-Lipschitz continuous}}%
\global\long\def\conv{\operatorname{conv}}%
\global\long\def\limout{\operatorname{LimOut}}%
\global\long\def\liminn{\operatorname{LimInn}}%
\global\long\def\gph{\operatorname{gph}}%
\global\long\def\textarrow#1{\stackrel{#1}{\longrightarrow}}%
\global\long\def\aleq{\overset{(a)}{\leq}}%
\global\long\def\aeq{\overset{(a)}{=}}%
\global\long\def\ageq{\overset{(a)}{\geq}}%
\global\long\def\bleq{\overset{(b)}{\leq}}%
\global\long\def\beq{\overset{(b)}{=}}%
\global\long\def\bgeq{\overset{(b)}{\geq}}%
\global\long\def\cleq{\overset{(c)}{\leq}}%
\global\long\def\ceq{\overset{(c)}{=}}%
\global\long\def\cgeq{\overset{(c)}{\geq}}%
\global\long\def\dleq{\overset{(d)}{\leq}}%
\global\long\def\deq{\overset{(d)}{=}}%
\global\long\def\dgeq{\overset{(d)}{\geq}}%
\global\long\def\eleq{\overset{(e)}{\leq}}%
\global\long\def\eeq{\overset{(e)}{=}}%
\global\long\def\egeq{\overset{(e)}{\geq}}%
\global\long\def\fleq{\overset{(f)}{\leq}}%
\global\long\def\feq{\overset{(f)}{=}}%
\global\long\def\fgeq{\overset{(f)}{\geq}}%
\global\long\def\gleq{\overset{(g)}{\leq}}%

%% file: main.bbl
\begin{thebibliography}{46}
\providecommand{\natexlab}[1]{#1}
\providecommand{\url}[1]{\texttt{#1}}
\expandafter\ifx\csname urlstyle\endcsname\relax
  \providecommand{\doi}[1]{doi: #1}\else
  \providecommand{\doi}{doi: \begingroup \urlstyle{rm}\Url}\fi

\bibitem[Allen-Zhu and Hazan(2016)]{allen2016variance}
Zeyuan Allen-Zhu and Elad Hazan.
\newblock Variance reduction for faster non-convex optimization.
\newblock In \emph{International conference on machine learning}, pages 699--707. PMLR, 2016.

\bibitem[Bagirov et~al.(2014)Bagirov, Karmitsa, and M{\"a}kel{\"a}]{bagirov2014introduction}
Adil Bagirov, Napsu Karmitsa, and Marko~M M{\"a}kel{\"a}.
\newblock \emph{Introduction to Nonsmooth Optimization: theory, practice and software}, volume~12.
\newblock Springer, 2014.

\bibitem[Balasubramanian and Ghadimi(2022)]{balasubramanian2022zeroth}
Krishnakumar Balasubramanian and Saeed Ghadimi.
\newblock Zeroth-order nonconvex stochastic optimization: Handling constraints, high dimensionality, and saddle points.
\newblock \emph{Foundations of Computational Mathematics}, 22\penalty0 (1):\penalty0 35--76, 2022.

\bibitem[Chen et~al.(2023{\natexlab{a}})Chen, Xu, and Luo]{chen2023faster}
Lesi Chen, Jing Xu, and Luo Luo.
\newblock Faster gradient-free algorithms for nonsmooth nonconvex stochastic optimization.
\newblock In \emph{International Conference on Machine Learning}, pages 5219--5233. PMLR, 2023{\natexlab{a}}.

\bibitem[Chen et~al.(2017)Chen, Zhang, Sharma, Yi, and Hsieh]{chen2017zoo}
Pin-Yu Chen, Huan Zhang, Yash Sharma, Jinfeng Yi, and Cho-Jui Hsieh.
\newblock Zoo: Zeroth order optimization based black-box attacks to deep neural networks without training substitute models.
\newblock In \emph{Proceedings of the 10th ACM workshop on artificial intelligence and security}, pages 15--26, 2017.

\bibitem[Chen et~al.(2023{\natexlab{b}})Chen, Zhou, Liang, and Lu]{chen2023generalized}
Ziyi Chen, Yi~Zhou, Yingbin Liang, and Zhaosong Lu.
\newblock Generalized-smooth nonconvex optimization is as efficient as smooth nonconvex optimization.
\newblock In \emph{International Conference on Machine Learning}, pages 5396--5427. PMLR, 2023{\natexlab{b}}.

\bibitem[Chezhegov et~al.(2024)Chezhegov, Klyukin, Semenov, Beznosikov, Gasnikov, Horv{\'a}th, Tak{\'a}{\v{c}}, and Gorbunov]{chezhegov2024gradient}
Savelii Chezhegov, Yaroslav Klyukin, Andrei Semenov, Aleksandr Beznosikov, Alexander Gasnikov, Samuel Horv{\'a}th, Martin Tak{\'a}{\v{c}}, and Eduard Gorbunov.
\newblock Gradient clipping improves adagrad when the noise is heavy-tailed.
\newblock \emph{arXiv preprint arXiv:2406.04443}, 2024.

\bibitem[Clarke(1990)]{clarke1990optimization}
Frank~H Clarke.
\newblock \emph{Optimization and nonsmooth analysis}.
\newblock SIAM, 1990.

\bibitem[Cui et~al.(2023)Cui, Shanbhag, and Yousefian]{cui2023complexity}
Shisheng Cui, Uday~V Shanbhag, and Farzad Yousefian.
\newblock Complexity guarantees for an implicit smoothing-enabled method for stochastic mpecs.
\newblock \emph{Mathematical Programming}, 198\penalty0 (2):\penalty0 1153--1225, 2023.

\bibitem[Cui and Pang(2021)]{cui2021modern}
Ying Cui and Jong-Shi Pang.
\newblock \emph{Modern nonconvex nondifferentiable optimization}.
\newblock SIAM, 2021.

\bibitem[Cutkosky and Mehta(2020)]{cutkosky2020momentum}
Ashok Cutkosky and Harsh Mehta.
\newblock Momentum improves normalized sgd.
\newblock In \emph{International conference on machine learning}, pages 2260--2268. PMLR, 2020.

\bibitem[Cutkosky and Orabona(2019)]{cutkosky2019momentum}
Ashok Cutkosky and Francesco Orabona.
\newblock Momentum-based variance reduction in non-convex sgd.
\newblock \emph{Advances in neural information processing systems}, 32, 2019.

\bibitem[Duchi et~al.(2012)Duchi, Bartlett, and Wainwright]{duchi2012randomized}
John~C Duchi, Peter~L Bartlett, and Martin~J Wainwright.
\newblock Randomized smoothing for stochastic optimization.
\newblock \emph{SIAM Journal on Optimization}, 22\penalty0 (2):\penalty0 674--701, 2012.

\bibitem[Fang et~al.(2018)Fang, Li, Lin, and Zhang]{fang2018spider}
Cong Fang, Chris~Junchi Li, Zhouchen Lin, and Tong Zhang.
\newblock Spider: Near-optimal non-convex optimization via stochastic path-integrated differential estimator.
\newblock \emph{Advances in neural information processing systems}, 31, 2018.

\bibitem[Flaxman et~al.(2005)Flaxman, Kalai, and McMahan]{flaxman2005online}
Abraham~D Flaxman, Adam~Tauman Kalai, and H~Brendan McMahan.
\newblock Online convex optimization in the bandit setting: gradient descent without a gradient.
\newblock In \emph{Proceedings of the sixteenth annual ACM-SIAM symposium on Discrete algorithms}, pages 385--394, 2005.

\bibitem[Gao and Deng(2024)]{gao2024stochastic}
Wenzhi Gao and Qi~Deng.
\newblock Stochastic weakly convex optimization beyond lipschitz continuity.
\newblock \emph{arXiv preprint arXiv:2401.13971}, 2024.

\bibitem[Ghadimi and Lan(2013)]{ghadimi2013stochastic}
Saeed Ghadimi and Guanghui Lan.
\newblock Stochastic first-and zeroth-order methods for nonconvex stochastic programming.
\newblock \emph{SIAM journal on optimization}, 23\penalty0 (4):\penalty0 2341--2368, 2013.

\bibitem[Grimmer(2019)]{grimmer2019convergence}
Benjamin Grimmer.
\newblock Convergence rates for deterministic and stochastic subgradient methods without lipschitz continuity.
\newblock \emph{SIAM Journal on Optimization}, 29\penalty0 (2):\penalty0 1350--1365, 2019.

\bibitem[H{\"u}bler et~al.(2024)H{\"u}bler, Fatkhullin, and He]{hubler2024gradient}
Florian H{\"u}bler, Ilyas Fatkhullin, and Niao He.
\newblock From gradient clipping to normalization for heavy tailed sgd.
\newblock \emph{arXiv preprint arXiv:2410.13849}, 2024.

\bibitem[Jain et~al.(2017)Jain, Kar, et~al.]{jain2017non}
Prateek Jain, Purushottam Kar, et~al.
\newblock Non-convex optimization for machine learning.
\newblock \emph{Foundations and Trends{\textregistered} in Machine Learning}, 10\penalty0 (3-4):\penalty0 142--363, 2017.

\bibitem[Kornowski and Shamir(2023)]{kornowski2023algorithm}
Guy Kornowski and Ohad Shamir.
\newblock An algorithm with optimal dimension-dependence for zero-order nonsmooth nonconvex stochastic optimization.
\newblock \emph{arXiv preprint arXiv:2307.04504}, 2023.

\bibitem[Larson et~al.(2019)Larson, Menickelly, and Wild]{larson2019derivative}
Jeffrey Larson, Matt Menickelly, and Stefan~M Wild.
\newblock Derivative-free optimization methods.
\newblock \emph{Acta Numerica}, 28:\penalty0 287--404, 2019.

\bibitem[Lei et~al.(2024)Lei, Pong, Sun, and Yue]{lei2024subdifferentially}
Ming Lei, Ting~Kei Pong, Shuqin Sun, and Man-Chung Yue.
\newblock Subdifferentially polynomially bounded functions and gaussian smoothing-based zeroth-order optimization.
\newblock \emph{arXiv preprint arXiv:2405.04150}, 2024.

\bibitem[Li et~al.(2024)Li, Qian, Tian, Rakhlin, and Jadbabaie]{li2024convex}
Haochuan Li, Jian Qian, Yi~Tian, Alexander Rakhlin, and Ali Jadbabaie.
\newblock Convex and non-convex optimization under generalized smoothness.
\newblock \emph{Advances in Neural Information Processing Systems}, 36, 2024.

\bibitem[Lin et~al.(2022)Lin, Zheng, and Jordan]{lin2022gradient}
Tianyi Lin, Zeyu Zheng, and Michael Jordan.
\newblock Gradient-free methods for deterministic and stochastic nonsmooth nonconvex optimization.
\newblock \emph{Advances in Neural Information Processing Systems}, 35:\penalty0 26160--26175, 2022.

\bibitem[Lin et~al.(2024)Lin, Xia, Deng, and Luo]{lin2024decentralized}
Zhenwei Lin, Jingfan Xia, Qi~Deng, and Luo Luo.
\newblock Decentralized gradient-free methods for stochastic non-smooth non-convex optimization.
\newblock In \emph{Proceedings of the AAAI Conference on Artificial Intelligence}, volume~38, pages 17477--17486, 2024.

\bibitem[Liu et~al.(2020)Liu, Chen, Kailkhura, Zhang, Hero~III, and Varshney]{liu2020primer}
Sijia Liu, Pin-Yu Chen, Bhavya Kailkhura, Gaoyuan Zhang, Alfred~O Hero~III, and Pramod~K Varshney.
\newblock A primer on zeroth-order optimization in signal processing and machine learning: Principals, recent advances, and applications.
\newblock \emph{IEEE Signal Processing Magazine}, 37\penalty0 (5):\penalty0 43--54, 2020.

\bibitem[Liu and Zhou(2024)]{liu2024nonconvex}
Zijian Liu and Zhengyuan Zhou.
\newblock Nonconvex stochastic optimization under heavy-tailed noises: Optimal convergence without gradient clipping.
\newblock \emph{arXiv preprint arXiv:2412.19529}, 2024.

\bibitem[Liu et~al.(2023)Liu, Jagabathula, and Zhou]{liu2023near}
Zijian Liu, Srikanth Jagabathula, and Zhengyuan Zhou.
\newblock Near-optimal non-convex stochastic optimization under generalized smoothness.
\newblock \emph{arXiv preprint arXiv:2302.06032}, 2023.

\bibitem[Lu(2019)]{lu2019relative}
Haihao Lu.
\newblock “relative continuity” for non-lipschitz nonsmooth convex optimization using stochastic (or deterministic) mirror descent.
\newblock \emph{INFORMS Journal on Optimization}, 1\penalty0 (4):\penalty0 288--303, 2019.

\bibitem[Mai and Johansson(2021)]{mai2021stability}
Vien~V Mai and Mikael Johansson.
\newblock Stability and convergence of stochastic gradient clipping: Beyond lipschitz continuity and smoothness.
\newblock In \emph{International Conference on Machine Learning}, pages 7325--7335. PMLR, 2021.

\bibitem[Malladi et~al.(2023)Malladi, Gao, Nichani, Damian, Lee, Chen, and Arora]{malladi2023fine}
Sadhika Malladi, Tianyu Gao, Eshaan Nichani, Alex Damian, Jason~D Lee, Danqi Chen, and Sanjeev Arora.
\newblock Fine-tuning language models with just forward passes.
\newblock \emph{Advances in Neural Information Processing Systems}, 36:\penalty0 53038--53075, 2023.

\bibitem[Mishkin et~al.(2024)Mishkin, Khaled, Wang, Defazio, and Gower]{mishkin2024directional}
Aaron Mishkin, Ahmed Khaled, Yuanhao Wang, Aaron Defazio, and Robert~M Gower.
\newblock Directional smoothness and gradient methods: Convergence and adaptivity.
\newblock \emph{arXiv preprint arXiv:2403.04081}, 2024.

\bibitem[Nelder and Mead(1965)]{nelder1965simplex}
John~A Nelder and Roger Mead.
\newblock A simplex method for function minimization.
\newblock \emph{The computer journal}, 7\penalty0 (4):\penalty0 308--313, 1965.

\bibitem[Nesterov and Spokoiny(2017)]{nesterov2017random}
Yurii Nesterov and Vladimir Spokoiny.
\newblock Random gradient-free minimization of convex functions.
\newblock \emph{Foundations of Computational Mathematics}, 17:\penalty0 527--566, 2017.

\bibitem[Nie(2009)]{nie2009sum}
Jiawang Nie.
\newblock Sum of squares method for sensor network localization.
\newblock \emph{Computational Optimization and Applications}, 43\penalty0 (2):\penalty0 151--179, 2009.

\bibitem[Reisizadeh et~al.(2023)Reisizadeh, Li, Das, and Jadbabaie]{reisizadeh2023variance}
Amirhossein Reisizadeh, Haochuan Li, Subhro Das, and Ali Jadbabaie.
\newblock Variance-reduced clipping for non-convex optimization.
\newblock \emph{arXiv preprint arXiv:2303.00883}, 2023.

\bibitem[Rubinstein and Kroese(2016)]{rubinstein2016simulation}
Reuven~Y Rubinstein and Dirk~P Kroese.
\newblock \emph{Simulation and the Monte Carlo method}.
\newblock John Wiley \& Sons, 2016.

\bibitem[Shamir(2017)]{shamir2017optimal}
Ohad Shamir.
\newblock An optimal algorithm for bandit and zero-order convex optimization with two-point feedback.
\newblock \emph{Journal of Machine Learning Research}, 18\penalty0 (1):\penalty0 1703--1713, 2017.

\bibitem[Sun et~al.(2024)Sun, Liu, and Yuan]{sun2024gradient}
Tao Sun, Xinwang Liu, and Kun Yuan.
\newblock Gradient normalization with (out) clipping ensures convergence of nonconvex sgd under heavy-tailed noise with improved results.
\newblock \emph{arXiv preprint arXiv:2410.16561}, 2024.

\bibitem[Tyurin(2024)]{tyurin2024toward}
Alexander Tyurin.
\newblock Toward a unified theory of gradient descent under generalized smoothness.
\newblock \emph{arXiv preprint arXiv:2412.11773}, 2024.

\bibitem[Wainwright(2019)]{wainwright2019high}
Martin~J Wainwright.
\newblock \emph{High-dimensional statistics: A non-asymptotic viewpoint}, volume~48.
\newblock Cambridge university press, 2019.

\bibitem[Zhang et~al.(2019)Zhang, He, Sra, and Jadbabaie]{zhang2019gradient}
Jingzhao Zhang, Tianxing He, Suvrit Sra, and Ali Jadbabaie.
\newblock Why gradient clipping accelerates training: A theoretical justification for adaptivity.
\newblock \emph{arXiv preprint arXiv:1905.11881}, 2019.

\bibitem[Zhang et~al.(2020)Zhang, Lin, Jegelka, Sra, and Jadbabaie]{zhang2020complexity}
Jingzhao Zhang, Hongzhou Lin, Stefanie Jegelka, Suvrit Sra, and Ali Jadbabaie.
\newblock Complexity of finding stationary points of nonconvex nonsmooth functions.
\newblock In \emph{International Conference on Machine Learning}, pages 11173--11182. PMLR, 2020.

\bibitem[Zhou et~al.(2020)Zhou, Xu, and Gu]{zhou2020stochastic}
Dongruo Zhou, Pan Xu, and Quanquan Gu.
\newblock Stochastic nested variance reduction for nonconvex optimization.
\newblock \emph{Journal of machine learning research}, 21\penalty0 (103):\penalty0 1--63, 2020.

\bibitem[Zhu et~al.(2023)Zhu, Zhao, and Zhang]{zhu2023unified}
Daoli Zhu, Lei Zhao, and Shuzhong Zhang.
\newblock A unified analysis for the subgradient methods minimizing composite nonconvex, nonsmooth and non-lipschitz functions.
\newblock \emph{arXiv preprint arXiv:2308.16362}, 2023.

\end{thebibliography}
